\pdfoutput=1
\documentclass[english]{jnsao}
\usepackage[utf8]{inputenc}

\usepackage[nameinlink,capitalize]{cleveref}

\newcommand{\Llog}{L\log L}
\newcommand{\Lexp}{L_{\exp}}
\newcommand{\bd}{\,{\mathrm{d}}}

\newcommand{\1}{{\mathbb{1}}}
\newcommand{\R}{{\mathbb{R}}}
\newcommand{\N}{{\mathbb{N}}}
\newcommand{\meas}{{\mathcal{M}}}
\newcommand{\lebesgue}{{\mathcal{L}}}
\newcommand{\prob}{{\mathcal{P}}}
\newcommand{\cont}{{\mathcal{C}}}

\newcommand{\pushforward}[2]{{{#1}_{\#}#2}}
\newcommand{\weakstarto}{\rightharpoonup^*}
\DeclareMathOperator*{\gammalim}{\Gamma\text{-}\lim}
\newcommand{\embed}{\hookrightarrow}
\newcommand{\norm}[1]{\|{#1}\|}

\newcommand{\abs}[1]{\left|{#1}\right|}

\DeclareMathOperator{\supp}{supp}

\usepackage[nameinlink,capitalize]{cleveref}

\usepackage{enumerate}
\usepackage{tikz}
\usepackage{pgfplots}
\pgfplotsset{compat=newest}
\pgfplotsset{plot coordinates/math parser=false}

\usepackage{mdframed}

\mdfdefinestyle{examplebg}{
    backgroundcolor=black!7,%
    hidealllines=true,%
    innertopmargin=.3em,%
    innerbottommargin=.7em,%
    innerleftmargin=.7em,%
    innerrightmargin=.7em,%
}

\surroundwithmdframed[style=examplebg]{example}

\def\thetitle{Entropic regularization of continuous optimal transport problems}

\title{\thetitle}
\author{Christian Clason%
    \thanks{%
        Faculty of Mathematics, 
        University of Duisburg-Essen, 
        45117 Essen, Germany 
        (\email{christian.clason@uni-due.de}, \orcid{0000-0002-9948-8426})
    }
    \and 
    Dirk A. Lorenz%
    \thanks{%
        Institute of Analysis and Algebra,
        TU Braunschweig,
        38092 Braunschweig, Germany
    (\email{d.lorenz@tu-braunschweig.de}, \orcid{0000-0002-7419-769X})}
    \and Hinrich Mahler%
    \thanks{%
        Institute of Analysis and Algebra,
        TU Braunschweig,
        38092 Braunschweig, Germany
    (\email{h.mahler@tu-braunschweig.de}, \orcid{0000-0001-9108-549X})}
    \and Benedikt Wirth%
    \thanks{%
        Applied Mathematics Münster,
        University of Münster,
        Einsteinstraße 62, 48149 Münster, Germany
        (\email{benedikt.wirth@uni-muenster.de}, \orcid{0000-0003-0393-1938})
    }
}
\shortauthor{Clason, Lorenz, Mahler, Wirth}

\date{2020-06-15}
\manuscriptlicense{}
\manuscriptcopyright{}
\manuscripteprinttype{arXiv}
\manuscripteprint{1906.01333v3}

\begin{document}
\maketitle

\begin{abstract}
    We analyze continuous optimal transport problems in the
    so-called Kantorovich form, where we seek a transport plan between two
    marginals that are probability measures on compact subsets of Euclidean
    space. We consider the case of regularization with the negative entropy
    with respect to the Lebesgue measure, which has attracted attention because
    it can be solved by the very simple Sinkhorn algorithm. We first analyze
    the regularized problem in the context of classical Fenchel duality and
    derive a strong duality result for a predual problem in the space of
    continuous functions. However, this problem may not admit a minimizer,
    which prevents obtaining primal-dual optimality conditions. We then show
    that the primal problem is naturally analyzed in the Orlicz space of
    functions with finite entropy in the sense that the entropically
    regularized problem admits a minimizer if and only if the marginals have
    finite entropy. We then derive a dual problem in the corresponding dual
    space, for which existence can be shown by purely variational arguments and
    primal-dual optimality conditions can be derived. For marginals that do not
    have finite entropy, we finally show Gamma-convergence of the regularized
    problem with smoothed marginals to the original Kantorovich problem.
\end{abstract}

\section{Introduction}

The Kantorovich formulation of optimal transport is the problem of finding a transport plan that describes how to move some measure onto another measure of the same mass such that a certain cost functional is minimal \cite{kantorovich1942masses}.
Specifically, let $\Omega_{1}$ and $\Omega_{2}$ be two compact subset of $\R^{n_{1}}$ and $\R^{n_{2}}$, respectively. For given probability measures $\mu$ on $\Omega_{1}$ and $\nu$ on $\Omega_{2}$ and a continuous cost function $c:\Omega_{1}\times\Omega_{2}\to [0,\infty)$, the goal is to find a measure $\pi$ on $\Omega_{1}\times\Omega_{2}$ such that the cost $\int_{\Omega_{1}\times \Omega_{2}}c\bd\pi$ is minimal among all $\pi$ that have $\mu$ and $\nu$ as marginals. This problem has been well studied, and we refer to the recent books \cite{villani2008optimal,Santambrogio} for an overview. For example, it is known that the problem has a solution $\pi$ and that the support of $\pi$ is contained in the so-called $c$-superdifferential of a $c$-concave function on $\Omega_{1}$, see \cite[Thm.~1.13]{ambrosio2013user}. (This is sometimes called the \emph{fundamental theorem of optimal transport}.) In the case where $\Omega_{1}$ and $\Omega_{2}$ are both subsets of $\R^{n}$ and where $c(x_1,x_2) = |x_1-x_2|^2$ is the squared Euclidean distance, this implies that optimal plans $\pi$ are singular with respect to the Lebesgue measure.
Hence, the optimal plan is not a measurable function, and so standard approximation techniques from numerical analysis (e.g. by piecewise constant or piecewise linear functions) are not applicable.
This motivates the use of regularization of the continuous problem to obtain approximate solutions that are functions instead of measures, which in turn can be treated by classical discretization techniques in order to solve the regularized problem.

In this work we focus on \emph{entropic regularization} by adding a multiple of the negative entropy of $\pi$ (with respect to the Lebesgue measure) to the objective function. This forces the optimal plan to be a measure that has a density with respect to the Lebesgue measure. Furthermore, in the discrete setting, this allows to solve the problem numerically by the very simple Sinkhorn algorithm \cite{Knight:2008,Cuturi:2013,Benamou:2015}.

\paragraph{Notation and problem statement.} To fully state the regularized optimal transport problem, we introduce some notation.
By $\meas(\Omega)$ and $\prob(\Omega)$ we denote the set of Radon and probability measures on $\Omega\subset\R^n$, respectively.
The Lebesgue measure will be denoted by $\lebesgue$ (the set on which it is defined being clear from the context), and integrals with respect to the Lebesgue measure are simply denoted by $\mathrm{d}x$ with the appropriate integration variable $x$.
We write $L^{p}(\Omega,\mathrm{d}\mu)$ for the space of $p$-integrable functions with respect to the measure $\mu$ but omit the set $\Omega$ if it is clear from the context. If no measure is given, $L^{p}$ always refers to the space with respect to the Lebesgue measure.
In the case where the measure $\pi$ has a density with respect to the Lebesgue measure, we will also use $\pi$ for that density.
For $\mu\in\meas(\Omega_{1})$ and $g:\Omega_{1}\to\Omega_{2}$, we denote by $\pushforward{g}{\pi}$ the pushforward of $\mu$ by $g$, i.e., the measure on $\Omega_{2}$ defined by $\pushforward{g}{\pi}(B) = \pi(g^{-1}(B))$ for all measurable sets $B\subset \Omega_2$.
In particular, we will use the coordinate projections $P_i:\Omega_1\times\Omega_2\to\Omega_i$, $P_i(x_1,x_2)=x_i$, and the fact that $\pushforward{P_{i}}{\pi}$ is the $i$th marginal of $\pi\in\meas(\Omega_1\times\Omega_2)$.
The entropically regularized Kantorovich problem of optimal mass transport between $\mu\in\prob(\Omega_1)$ and $\nu\in\prob(\Omega_2)$ is then given by
\begin{equation}\label{eq:entropyRegularized-intro}
    \inf_{\substack{\pi\in\prob(\Omega_1\times\Omega_2),\\ \pushforward{(P_1)}{\pi}=\mu,\,\pushforward{(P_2)}{\pi}=\nu}}
    \int_{\Omega_1\times\Omega_2}c\bd\pi+\gamma\int_{\Omega_1\times\Omega_2}\pi(\log\pi-1)\bd(x_1,x_2).
    \tag{P}
\end{equation}
(Note that we used the negative entropy of $\pi$ with respect to the Lebesgue measure for regularization.
One could also consider regularization by adding $\gamma\int_{\Omega_1\times\Omega_2}\pi(\log\pi-1)\bd \theta$ for some other measure $\theta$, e.g., the product measure $\mu\otimes\nu$ \cite{cuturi2018semidual}, but we will not pursue this further.)
A purely formal application of convex duality then yields the predual problem
\begin{equation}\label{eq:entropyRegularizedDual-formal-intro}
    \sup_{\substack{\alpha\in \cont(\Omega_{1})\\\beta\in\cont(\Omega_{2})}}\int_{\Omega_{1}}\alpha(x_1)\bd\mu(x_1)+\int_{\Omega_{2}}\beta(x_2)\bd\nu(x_2)-\gamma\int_{\Omega_1\times\Omega_2}\exp\left(\tfrac{-c(x_1,x_2)+\alpha(x_2)+\beta(x_1)}\gamma\right)\bd(x_1,x_2).
    \tag{D}
\end{equation}
Having a primal and a dual problem, it is now possible to write down the system of Fenchel--Rockafellar extremality conditions and derive and analyze algorithms to solve this system; in fact, this is one of the possible ways of deriving the Sinkhorn algorithm in the discrete case.
However, the existence of solutions to \eqref{eq:entropyRegularizedDual-formal-intro} -- which is necessary to rigorously obtain extremality conditions -- is not obvious in the continuous case. As it turns out, neither \eqref{eq:entropyRegularized-intro} nor \eqref{eq:entropyRegularizedDual-formal-intro} may admit a solution in the considered spaces. As we will show, it is necessary and sufficient to obtain existence of a primal solution for the marginals to be in the Banach space $\Llog$ of functions of finite entropy; correspondingly, a reformulation of the predual problem in the \emph{dual} space $\Lexp$ allows showing existence of a maximizer by purely variational methods. For marginals that are not in $\Llog$, we show $\Gamma$-convergence of minimizers of regularized problems with suitably smoothed marginals.

\paragraph{Related work.} The continuous optimal transport problem has been analyzed in the survey paper \cite{Leonard:2014} where the relation to the so-called dynamic Schrödinger problem is made.
Another survey \cite{Essid:2019} presents an existence proof for a reparameterized optimality system based on the convergence analysis for a continuous variant of Sinkhorn's algorithm (and attributes the proof and the algorithm  to Fortet \cite{Fortet:1940}).
A detailed overview of the connections between optimal transport, the Schrödinger problem, and the Sinkhorn algorithm from a stochastic control viewpoint is given in the even more recent survey \cite{Chen:2020}.
In \cite{carlier2017convergence}, primal existence has been shown in the subset of the space of measures which have a density of finite entropy with respect to the Lebesgue measure.
Furthermore, \cite{Chizat:2018} analyzes the problem (for unbalanced transport, i.e., for marginals with different mass) in $L^{1}$ and derives a dual formulation in $L^{\infty}$. However, the question of existence of a solution of the respective dual problem is not answered. In \cite{Borwein:1994}, this gap was closed through a contraction argument using the Hilbert metric. More precisely, \cite[Thm. 3.1]{Borwein:1994} guarantees the existence of dual solutions in $L^\infty$ provided that the feasible set of the dual problem is not empty. Moreover, if a certain constraint qualification holds, the dual optimizers $x$ and $y$ can be shown to satisfy
$x(s) + y(t) = \log u_0(s,t) \quad\text{a.e.}$,
where $u_0$ denotes the optimal primal solution.
Here $u_0$, $x$, and $y$ correspond to $\bar\pi$, $\alpha$, $\beta$ of \eqref{eq:entropyRegularized-intro} and \eqref{eq:entropyRegularizedDual-formal-intro} via
$u_0 = e^{\frac c\gamma}\bar{\pi}$, $x = \alpha/\gamma$, and $y = \beta/\gamma$.
A similar result is also stated in the more recent work \cite[Thm.~6]{Chen:2015}, which shows the existence of dual optimizers if the marginals are absolutely continuous probability measures. (The relation of $q$ and $\nu$ in \cite{Chen:2015} to our notation is $q = e^{\frac{-c}{\gamma}}$ and $\nu = \bar\pi e^{\frac c \gamma}$.)
Another approach to prove the existence of unique solutions (even in the multi-marginal case) is presented in \cite[Thm 4.3]{Carlier:2018}. The authors show that a certain map is a bijection, which yields existence of dual solutions $\alpha$ and $\beta$ in $L^\infty$ if the marginals are functions in $L^\infty$ as well. Moreover, in \cite{Berman:2017} a compactness argument is used to show the existence of a fixed point of the Sinkhorn iteration; in contrast to our work, the entropy penalization there is considered with respect to the product measure of the marginals.

Previous works \cite{Berman:2017,Borwein:1994,Carlier:2018,Chen:2015} tackle the problem of existence of dual solutions under various conditions in standard Lebesgue spaces. For marginals of finite entropy, \cite[Cor. 3.2]{Csiszar:1975} already states that dual solutions exist and satisfy $\bar\alpha\in L^1(\Omega_1,\mu)$ and $\bar\beta\in L^1(\Omega_2,\nu)$ (in our notation; the notation there uses $Q=\bar\pi e^{\frac{-c}{\gamma}}$, $P_1=\mu$, $a=e^{\frac{-\alpha}\gamma}$, $R=e^{\frac{-c}\gamma}\lebesgue$, $P_2=\nu$, and $b=e^{\frac{-\beta}\gamma}$). Note that while the primal solution $Q$ is in $\Llog(\Omega, \lebesgue)$, the analysis takes place in $\Llog(\Omega, e^{\frac{-c}\gamma}\lebesgue)$.
Moreover, as the authors of \cite{Borwein:1994} note, \cite{Csiszar:1975} fails to elaborate a crucial step of the argumentation. This gap was closed only later in \cite{Borwein:1992c}.
None of the mentioned works considered necessary conditions for existence.
Finally, \cite{Lorenz:2019} analyzes regularization with the $L^{2}$ norm of $\pi$ and derives existence of solutions of the dual problem.

The notion of Orlicz spaces in the context of convex integral functionals has previously been used in \cite{Leonard:2008}, where existence of both primal and dual optimizers are covered in a more general setting. 
More precisely, the spaces used in \cite{Leonard:2008}, which are also known as Musielak--Orlicz spaces \cite{musielak:1983}, are a generalization of the Orlicz spaces used here. 
The setting considered here can be recovered in two different ways:
In section 7.3 (a), the above referenced results of \cite{Csiszar:1975} are recovered as a special case (where again our case corresponds to choosing $R = e^{\frac{-c}\gamma}\lebesgue$).
Moreover, choosing $m(z) = e^{\frac{-c(z)}{\gamma}}$ in the second example (titled \emph{a variant of the Boltzmann entropy}) in section 7.1 gives a problem very similar to the one considered here. The difference lies in the fact that the cost function $c$ is part of the definition of the relevant Musielak--Orlicz spaces in this case, and hence the analysis takes place in different spaces.
As the aim of \cite{Leonard:2008} is to weaken the necessary assumptions as much as possible, the overall setting is more abstract, and the proofs rely heavily on the authors previous work \cite{Leonard:2010}. Here we aim for a self-contained, more elementary, treatment of \eqref{eq:entropyRegularized-intro}.

Regarding $\Gamma$-convergence, the limit for $\gamma\to 0$ and fixed marginals with densities with finite entropy was considered recently in \cite{carlier2017convergence}.

\paragraph{Organization.} The next \cref{sec:finite-entropy} recalls statements about functions of finite entropy and the duality of the respective Orlicz space $\Llog$. In \cref{sec:fenchel-duality}, we collect and prove (for the sake of completeness) results on the regularized optimal transport problem \eqref{eq:entropyRegularized} in the context of duality of continuous functions and measures. In particular, \cref{thm:solutionProperties} shows that primal solutions exist if and only if the marginals are in the space $\Llog$. Hence, we analyze the problem in \cref{sec:duality-llogl-lexp} in the context of $\Llog$ and $\Lexp$. We show existence and uniqueness of the primal problem in $\Llog$, derive the dual problem and show existence of solutions for the dual problem in $\Lexp$. We finally show a result on $\Gamma$-convergence for the combined regularization and smoothing of marginals that do not have finite entropy in \cref{sec:gamma}. 

\section{Review of functions of finite entropy and the space \texorpdfstring{$\scriptstyle\Llog$}{LlogL}}
\label{sec:finite-entropy}

Entropic regularization deals with positive integrable functions of finite entropy.
These functions are closely connected to the space $\Llog$, a special case of (Birnbaum--)Orlicz spaces, and hence we collect some facts about this space which are mainly taken from \cite{Rao:1988,Bennett:1988,adams:2003}; see also \cite{simon2011convexity}.
We consider a compact domain $\Omega\subset\R^n$ and
denote the \emph{neg-entropy} of a measurable function $f:\Omega\to \R$ by
\begin{equation*}
    E(f) = \int_\Omega \abs{f(x)}\log(\abs{f(x)})\bd x,
\end{equation*}
where we set $0\log 0 = 0$ as usual.
Note that since $s\log s\geq -1/e$ for every $s\geq 0$, the neg-entropy always lies in the interval $[-\lebesgue(\Omega)/e,\infty]$.
We say that $f$ has finite entropy if $E(f)<\infty$.
Following \cite{Bennett:1988}, we define
\begin{equation*}
    \Llog(\Omega) := \left\{f:\Omega\to \R \text{ measurable} : \int_\Omega \abs{f(x)}\log^+(\abs{f(x)})\bd x<\infty\right\},
\end{equation*}
where $\log^+(x) = \max(\log(x),0)$.
\begin{proposition}[\protect{\cite[Thm.~1.2]{Navrotskaya:2013}}]\label{thm:finiteEntropy}
    A nonnegative measurable function $f$ on a set with finite measure has finite entropy if and only if $f\in \Llog(\Omega)$.
\end{proposition}
It turns out that $\Llog(\Omega)$ can be normed such that it becomes a Banach space and that its dual has a natural characterization.
In the following, we recall the central constructions and main results based on so-called Young functions.
\begin{definition}[Young functions]
    Let $\varphi : [0, \infty) \to [0, \infty]$ be increasing and lower semi-continuous with $\varphi(0) = 0$. Suppose that $\varphi$ is neither identically zero nor identically infinite on $(0,\infty)$. Then the function $\Phi$, defined by
    \begin{equation*}
        \Phi(t) := \int_0^t \varphi(s) \bd s\,,
    \end{equation*}
    is said to be a \emph{Young function}. Moreover, the function $\Psi$ defined by
    \begin{equation*}
        \Psi(s) := \max_{t\geq0} \{st - \Phi(t)\}
    \end{equation*}
    is called the \emph{complementary Young function} of $\Phi$.
\end{definition}

Any Young function is continuous and convex on its domain, and the complementary Young function $\Psi$ is again a Young function.
The notion of Young functions gives rise to a generalization of $L^p$ spaces through the definition of the so-called Luxemburg norm.
\begin{definition}[Luxemburg norm and Orlicz spaces]
    Let $\Phi$ be a Young function. The Luxemburg norm of a measurable function $f:\Omega\to\R$ is defined as
    \begin{equation}\label{eq:luxemburg}
        \|f\|_{\Phi}=\inf\left\{\gamma > 0: \int_\Omega\Phi\left(\frac{|f|}\gamma\right)\bd x\leq1\right\}\,.
    \end{equation}
    The space of all measurable functions with finite Luxemburg norm is called \emph{Orlicz space} and denoted by $L^\Phi(\Omega)$.
\end{definition}

\begin{remark}\label{rem:orlicz}
    General Orlicz norms do not scale in a simple way with the size of the set $\Omega$.  Writing $\1_A$ for the characteristic function of the set $A\subset \Omega$, i.e., $\1_A(x) = 1$ if $x\in A$ and $0$ else, the  $p$-norm (corresponding to the Young function $\Phi(t)=t^{p}$) of  $\1_{\Omega}$ equals $\norm{\1_{\Omega}}_{p} = \lebesgue(\Omega)^{1/p}$. For a strictly increasing Young function $\Phi$, we obtain the more complicated result $\norm{\1_{\Omega}}_{\Phi} = (\Phi^{-1}(\lebesgue(\Omega)^{-1}))^{-1}$. As a consequence, some results in the following depend on the size of the domain. One could get rid of this dependence by adapting the definition of the norm to, e.g.,
    \[
        \|f\|_{\Phi}=\inf\left\{\gamma> 0: \frac1{\lebesgue(\Omega)}\int_\Omega\Phi\left(\frac{|f|}\gamma\right)\bd x\leq1\right\}\,.
    \]
    However, since this
    definition would be nonstandard, we refrain from doing so.

    Moreover, note that
    \[
        \int_\Omega\Phi\left(\frac{|f|}{\|f\|_{\Phi}}\right)\bd x\leq1
    \]
    is always true, but equality may fail to hold. For a counterexample, see, e.g., \cite[Example 2.8]{simon2011convexity}.
\end{remark}
\begin{theorem}[\protect{\cite[Thm.~8.10]{adams:2003}}]
    $L^\Phi(\Omega)$ is a Banach space with respect to the Luxemburg norm.
\end{theorem}

We will also need the following estimate.
\begin{lemma}\label{thm:OrliczFunctionEstimate}
    Let $L^\Phi(\Omega)$ denote the Orlicz space with convex Young function $\Phi$ and $u\in L^\Phi(\Omega)$ with $\norm{u}_{\Phi}>1$. Then $\int_\Omega\Phi(|u|)\bd x\geq\norm{u}_{\Phi}$.
\end{lemma}
\begin{proof}
    For any $1\leq\gamma<\norm{u}_{\Phi}$, it holds that $\int_{\Omega}\Phi\big(\frac{|u|}{\gamma}\big)\bd x>1$.
    It then follows from the convexity of $\Phi$ and $\Phi(0)=0$ that
    \begin{equation*}
        \begin{aligned}
            \frac{1}{\gamma}\int_{\Omega}\Phi(|u|)\bd x & = \int_{\Omega}\tfrac1\gamma\Phi(|u|) + \left(1-\tfrac1\gamma\right)\Phi(0)\bd x\\
            & \geq \int_{\Omega}\Phi\left(\tfrac1\gamma |u| + \left(1-\tfrac1\gamma\right)0\right)\bd x= \int_{\Omega}\Phi\left(\frac{|u|}{\gamma}\right)\bd x>1.
        \end{aligned}
    \end{equation*}
    Letting $\gamma\to\norm{u}_{\Phi}$, the claim follows.
\end{proof}

Note that by \cref{rem:orlicz}, \cref{thm:OrliczFunctionEstimate} does \emph{not} hold for $\norm{u}_\Phi = 1$.

Using $\Phi_{\log}(s)=s\log^+s$ as Young function now immediately yields $\Llog(\Omega) = L^{\Phi_{\log}}(\Omega)$. The complementary Young function
\begin{equation}\label{eq:Phiexp}
    \Phi_{\exp}(s)=
    \begin{cases}
        s& \text{if } 0<s\leq 1,\\
        e^{s-1}& \text{if } s>1
    \end{cases}
\end{equation}
of $\Phi_{\log}$ now provides a natural way to define the Orlicz space $L^{\Phi_{\exp}}(\Omega)=:L_{\exp}(\Omega)$. In fact, $L_{\exp}(\Omega)$ is the dual space of $\Llog(\Omega)$. 

\begin{proposition}[\protect{\cite[Thm.~IV.6.5]{Bennett:1988}}]
    \label{thm:dual-LlogL}
    If $\Omega$ has finite Lebesgue measure, then $\Llog(\Omega)^* = \Lexp(\Omega)$ (up to equivalence of norms).
    Moreover, for all $1<p<\infty$, the following embeddings hold
    \begin{equation*}
        L^\infty(\Omega)\embed \Lexp(\Omega)\embed L^p(\Omega)\embed \Llog(\Omega)\embed L^1(\Omega).
    \end{equation*}
\end{proposition}
The Luxemburg norms \eqref{eq:luxemburg} on $L^\Phi(\Omega)$ are equivalent to the norms defined in \cite[Def.~IV.6.3]{Bennett:1988}
(in \cite[Def.~IV.6.3]{Bennett:1988}, the norms for $\Llog(\Omega)$ and $\Lexp(\Omega)$ are dual to each other).
The constants in this norm equivalence will in the following generically be denoted by $c_{\Phi}$.
Note that \cite[Thm.~IV.6.5]{Bennett:1988} is stated for domains with unit Lebesgue measure, but the case of general finite measure follows by a simple rescaling.

We also have the following properties, which follow from Theorem 8.21\,b and Theorem 8.19 in \cite{adams:2003}, respectively, by observing that $\Phi_{\log}$ is so-called $\Delta$-regular (c.f. \cite[Def. 8.7]{adams:2003}) but $\Phi_{\exp}$ is not.
\begin{lemma}
    \label{lem:llogl-separable}
    ~
    \begin{enumerate}[(i)]
        \item The space $\Llog(\Omega)$ is separable.
        \item The spaces $\Llog(\Omega)$ and $\Lexp(\Omega)$ are not reflexive.
    \end{enumerate}
\end{lemma}

The following example shows that the desired optimality conditions cannot be derived by simply setting the Gâteaux derivative to zero.
\begin{example}\label{ex:LlogL-log-not-Lexp}
    $E$ is not Gâteaux-differentiable on $\Llog((0,1))$. Indeed,
    consider $f(x) = \exp(-1/\sqrt{x})$. 
    Then it holds that
    $f\in \Llog((0,1))$ (since $f$ is bounded) and hence that $E(f)<\infty$, but note that
    the formal Gâteaux derivative $E'(f) = \log(f)+1$ is not in $\Lexp((0,1))$.
    To see this, note that $\log(f) + 1 = 1-\frac1{\sqrt x}$ is not in $L^2((0,1))$ and thus by \cref{thm:dual-LlogL} is not in $\Lexp((0,1))$.
\end{example}

We next derive a few facts that will be useful for the analysis of the primal and dual regularized optimal transport problems.
For the first lemma, we use the elementary fact that for all $a,b>0$ we have $\log^+(ab)\leq\log^+(a)+\log^+(b)$.
\begin{lemma}\label{thm:sum}
    If $\mu\in \Llog(\Omega_1)$, $\nu\in \Llog(\Omega_2)$, 
    and $\pi = \mu\otimes \nu$ (i.e., $\pi(x_1,x_2):=\mu(x_1)\nu(x_2)$),
    then $\pi\in \Llog(\Omega_1\times\Omega_2)$.
\end{lemma}
\begin{proof}
    We simply estimate
    \begin{multline*}
        \int_{\Omega_1\times\Omega_2}|\mu(x_1)\nu(x_2)|\log^+|\mu(x_1)\nu(x_2)|\bd (x_1,x_2)\\
        \leq\int_{\Omega_1}|\mu(x_1)|\log^+|\mu(x_1)|\bd x_1\int_{\Omega_2}|\nu(x_2)|\bd x_2
        +\int_{\Omega_1}|\mu(x_1)|\bd x_1\int_{\Omega_2}|\nu(x_2)|\log^+|\nu(x_2)|\bd x_2
    \end{multline*}
    and use that all terms on the right-hand side are finite since $\Llog(\Omega)\subset L^1(\Omega)$.
\end{proof}

Next, we consider a function $\pi\in \Llog(\Omega_1\times\Omega_2)$ and
its pushforwards under the coordinate projections
\begin{equation*}
    \pushforward{(P_1)}\pi(x_1) = \int_{\Omega_{2}}\pi(x_1,x_2)\bd x_2,\qquad
    \pushforward{(P_2)}\pi(x_2) = \int_{\Omega_{1}}\pi(x_1,x_2)\bd x_1.
\end{equation*}
The following result states that these marginals are also in $\Llog$.
\begin{lemma}\label{thm:projection}
    If $\pi\in \Llog(\Omega_1\times\Omega_2)$, then $\pushforward{(P_i)}\pi\in \Llog(\Omega_i)$ for $i\in\{1,2\}$ with 
    \begin{equation*}
        \norm{\pushforward{(P_{i})}\pi}_{\Phi_{\log}}\leq\max(1,\lebesgue(\Omega_{3-i}))\norm{\pi}_{\Phi_{\log}}.
    \end{equation*}
\end{lemma}
\begin{proof}
    Using the convexity of $\Phi(s) = s\log^{+}(s)$ and Jensen's inequality, we obtain 
    \begin{equation*}
        \int_{\Omega_1\times\Omega_2}\Phi\Big(\tfrac{|\pi|}{\gamma}\Big)\bd(x_{1},x_{2}) \geq \lebesgue(\Omega_1)\int_{\Omega_2}\Phi\Big(\frac1{\lebesgue(\Omega_1)}\int_{\Omega_1}\tfrac{|\pi|}{\gamma}\bd x_{1}\Big)\bd x_{2}
        \geq\int_{\Omega_2}\Phi\Big(\int_{\Omega_1}\tfrac{|\pi|}{\gamma\max(1,\lebesgue(\Omega_1))}\bd x_{1}\Big)\bd x_{2}
    \end{equation*} 
    where we used $\ell\Phi(s/\ell)\geq\Phi(s)$ for $\ell\leq1$ and $\ell\Phi(s/\ell)\geq\Phi(s/\ell)$ otherwise.
    Thus we obtain
    \begin{equation*}
        \begin{aligned}
            \norm{\pi}_{\Phi_{\log}} & = \min\left\{\gamma\geq 0: \int_{\Omega_1\times\Omega_2}\Phi\Big(\tfrac{|\pi|}{\gamma}\Big)\bd x_{1}\bd x_{2}\leq 1\right\}\\
            & \geq \min\left\{\gamma\geq 0: \int_{\Omega_2}\Phi\Big(\int_{\Omega_1}\tfrac{|\pi|}{\gamma\max(1,\lebesgue(\Omega_1))}\bd x_{1}\Big)\bd x_{2}\leq 1\right\}\\
            & = \min\left\{\gamma\geq 0: \int_{\Omega_2}\Phi\Big(\tfrac{\pushforward{(P_{2})}|\pi|}{\gamma\max(1,\lebesgue(\Omega_1))}\Big)\bd x_{2}\leq1\right\}\\
            & \geq \min\left\{\gamma\geq 0: \int_{\Omega_2}\Phi\Big(\tfrac{\pushforward{(P_{2})}\pi}{\gamma\max(1,\lebesgue(\Omega_1))}\Big)\bd x_{2}\leq1\right\}\\
            & = \frac{\norm{\pushforward{(P_{2})}\pi}_{\Phi_{\log}}}{\max(1,\lebesgue(\Omega_1))}.
        \end{aligned}
    \end{equation*}
    The claim for $\pushforward{(P_{1})}{\pi}$ follows similarly.
\end{proof}

As a corollary, we obtain a characterization of $\Lexp(\Omega)$ on tensor product spaces. 
\begin{corollary}\label{cor:Lexpoplus}
    It holds that $\alpha\in \Lexp(\Omega_1)$ and $\beta\in \Lexp(\Omega_2)$ if and only if $\alpha\oplus\beta\in \Lexp(\Omega_1\times\Omega_2)$, where
    \begin{equation*}
        (\alpha\oplus\beta)(x_{1},x_{2}) := \alpha(x_{1}) + \beta(x_{2}).
    \end{equation*}
\end{corollary}
\begin{proof}
    The mapping $(\alpha,\beta)\mapsto \alpha\oplus\beta$ is the adjoint
    of $\pi\mapsto(\pushforward{(P_1)}\pi,\pushforward{(P_2)}\pi)$, and hence one implication follows from the fact that $\Llog(\Omega)^{*} = \Lexp(\Omega)$.

    For the other implication, we use the Luxemburg norm and Jensen's inequality with $\Phi\equiv\Phi_{\exp}$ to observe that
    \begin{equation*}
        \begin{aligned}
            \norm{\alpha\oplus\beta}_{\Phi_{\exp}} & = \min\left\{\gamma\geq 0:\int_{\Omega_1\times\Omega_2}\Phi\Big(\tfrac{\alpha(x_{1})+\beta(x_{2})}{\gamma}\Big)\bd x_{1}\bd x_{2}\leq 1\right\}\\
            & \geq \min\left\{\gamma\geq 0:\lebesgue(\Omega_1)\int_{\Omega_2}\Phi\Big(\tfrac{\frac1{\lebesgue(\Omega_1)}\int_{\Omega_1}\alpha(x_{1})\bd x_{1}+\beta(x_{2})}{\gamma}\Big)\bd x_{2}\leq 1\right\}\\
            & \geq \min\left\{\gamma\geq 0:\int_{\Omega_2}\Phi\Big(\min(1,\lebesgue(\Omega_1))\tfrac{\frac1{\lebesgue(\Omega_1)}\int_{\Omega_1}\alpha(x_{1})\bd x_{1}+\beta(x_{2})}{\gamma}\Big)\bd x_{2}\leq 1\right\}\\
            & =\min(1,\lebesgue(\Omega_1)) \left\|\beta + {\textstyle\tfrac1{\lebesgue(\Omega_1)}\int_{\Omega_1}}\alpha\bd x_1\right\|_{\Phi_{\exp}}.
        \end{aligned}
    \end{equation*}
    This shows that $\beta$ plus a constant is in $\Lexp(\Omega)$ and hence that $\beta$ itself is in $\Lexp(\Omega)$.
    Arguing similarly for $\alpha$, we obtain the claim.
\end{proof}

\section{Fenchel duality in \texorpdfstring{$\scriptstyle\meas$}{M} and \texorpdfstring{$\scriptstyle\cont$}{C}}\label{sec:fenchel-duality}

In this section, we study the primal and dual problems for entropically regularized mass transport, i.e.,
\begin{equation}\label{eq:entropyRegularized}
    \inf_{\substack{\pi\in\prob(\Omega_1\times\Omega_2),\\ \pushforward{(P_1)}{\pi}=\mu,\,\pushforward{(P_2)}{\pi}=\nu}}
    \int_{\Omega_1\times\Omega_2}c\bd\pi+\gamma\int_{\Omega_1\times\Omega_2}\pi(\log\pi-1)\bd(x_1,x_2)
    \tag{P}
\end{equation}
and
\begin{equation}\label{eq:entropyRegularizedDual-formal}
    \sup_{\substack{\alpha\in \cont(\Omega_{1})\\\beta\in\cont(\Omega_{2})}}\int_{\Omega_{1}}\alpha(x_1)\bd\mu(x_1)+\int_{\Omega_{2}}\beta(x_2)\bd\nu(x_2)-\gamma\int_{\Omega_1\times\Omega_2}\exp\left(\tfrac{-c(x_1,x_2)+\alpha(x_2)+\beta(x_1)}\gamma\right)\bd(x_1,x_2),
    \tag{D}
\end{equation}
using Fenchel duality in the canonical spaces $\meas(\Omega_1\times\Omega_2)$ and $\cont(\Omega_1)\times\cont(\Omega_2)$. Most of the results in this section are classical~\cite{Csiszar:1975,Borwein:1994}, but we include the results with proofs for the sake of completeness.

We use the general framework as outlined in, e.g.,  \cite[Sec.~III.4]{Ekeland:1999} or \cite[Chap.~9]{Attouch:2006}. All throughout the following, we assume that $\mu\in \prob(\Omega_1)$, $\nu\in \prob(\Omega_2)$, $c\in\cont(\Omega_1\times\Omega_2)$, $\gamma>0$, and that  $\Omega_1$ and $\Omega_2$ are compact.

We begin with a strong duality result for \eqref{eq:entropyRegularized} and \eqref{eq:entropyRegularizedDual-formal}.  
A similar result in $L^{1}(\Omega)$ instead of $\meas(\Omega)$ is \cite[Thm.~3.2]{Chizat:2018}, but we state the theorem and its proof because we use a slightly different setting.
\begin{proposition}[strong duality]\label{thm:primalExistence}
    The predual problem to \eqref{eq:entropyRegularized} is \eqref{eq:entropyRegularizedDual-formal},
    and strong duality holds.
    Furthermore, if the supremum in \eqref{eq:entropyRegularizedDual-formal} is finite, \eqref{eq:entropyRegularized} admits a minimizer.
\end{proposition}
\begin{proof}
    First, by the Riesz--Markov representation theorem, $\meas(\Omega)$ is the dual space of $\cont(\Omega)$ for compact $\Omega$.
    Furthermore, Slater's condition is fulfilled with $\alpha,\beta=0$ so that strong duality holds and -- assuming a finite supremum -- the primal problem \eqref{eq:entropyRegularized} possesses a minimizer. 
    In addition, the integrand of the last integral in \eqref{eq:entropyRegularizedDual-formal} is normal so that it can be conjugated pointwise \cite{Rockafellar:1968}.
    Carrying out the conjugation, we obtain
    \begin{align*}
        &\sup_{\substack{\alpha\in\cont(\Omega_{1})\\\beta\in\cont(\Omega_{2})}}\int_{\Omega_{1}}\alpha\bd\mu+\int_{\Omega_{2}}\beta\bd\nu-\gamma\int_{\Omega_1\times\Omega_2}\exp\left(\tfrac{-c(x_1,x_2)+\alpha(x_2)+\beta(x_1)}\gamma\right)\bd(x_1,x_2)\\
        &=\sup_{\substack{\alpha\in\cont(\Omega_{1})\\\beta\in\cont(\Omega_{2})}}
        \int_{\Omega_1}\!\alpha\bd\mu+\!\int_{\Omega_2}\!\beta\bd\nu
        +\!\int_{\Omega_1\times\Omega_2}\!\min_{\pi\geq0}(c(x_1,x_2)-\alpha(x_1)-\beta(x_2))\pi(x_1,x_2)+\gamma\pi(\log\pi-1)\bd(x_1,x_2)\\
        &=\sup_{\substack{\alpha\in\cont(\Omega_{1})\\\beta\in\cont(\Omega_{2})}}\min_{\substack{\pi\in\meas(\Omega_1\times\Omega_2)\\\pi\geq0}}
        \int_{\Omega_1\times\Omega_2}\!c\pi+\gamma\pi(\log\pi-1)\bd(x_1,x_2)+\!\int_{\Omega_1}\!\alpha\bd(\mu\!-\!\pushforward{(P_1)}\pi)+\!\int_{\Omega_2}\!\beta\bd(\nu\!-\!\pushforward{(P_2)}\pi)\\
        &=\min_{\substack{\pi\in\prob(\Omega_1\times\Omega_2)\\
                \pushforward{(P_1)}{\pi}=\mu,\,
        \pushforward{(P_2)}{\pi}=\nu}}
        \int_{\Omega_1\times\Omega_2}c\bd\pi+\gamma\int_{\Omega_1\times\Omega_2}\pi(\log\pi-1)\bd (x_1,x_2)
        \,,
    \end{align*}
    which is \eqref{eq:entropyRegularized}.
\end{proof}

\begin{remark}\label{rem:existence}
    Note that \cref{thm:primalExistence} does \emph{not} claim that the supremum is attained, i.e., that the predual problem \eqref{eq:entropyRegularizedDual-formal} admits a solution.
    The proposition should also be compared to \cite[Thm.~3.2]{Chizat:2018}, which similarly characterizes solutions under the condition that the dual problem attains a maximizer.

    In addition, solutions to \eqref{eq:entropyRegularizedDual-formal} cannot be unique since we can add and subtract constants to $\alpha$ and $\beta$, respectively, without changing the functional value. On the other hand, up to such a constant, the functional in \eqref{eq:entropyRegularizedDual-formal} is strictly concave, and therefore any solution is uniquely determined by this constant.
\end{remark}

We can use this duality argument in combination with the results of \cref{sec:finite-entropy} to address the question of existence of a solution to \eqref{eq:entropyRegularized}. (Naturally, existence under the stated condition can also be shown using Tonelli's direct method; here we give a proof based on the already shown convex duality for the sake of conciseness.)
\begin{theorem}\label{thm:solutionProperties}
    Problem \eqref{eq:entropyRegularized} admits a minimizer $\bar \pi$
    if and only if $\mu\in \Llog(\Omega_1)$ and $\nu\in \Llog(\Omega_2)$. In this case, the minimizer is unique and lies in $\Llog(\Omega_1\times\Omega_2)$.
\end{theorem}
\begin{proof}
    By \cref{thm:finiteEntropy}, the energy is bounded if and only if $\bar \pi \in \Llog(\Omega_1\times\Omega_2)$. However, by \cref{thm:projection}, this is the case only if $\mu = \pushforward{(P_{1})}{\bar\pi} \in \Llog(\Omega_1)$ and similarly for $\nu$.
    This shows that the conditions are necessary to have a finite energy.
    For sufficiency,
    we first note that for $\mu\in \Llog(\Omega_1)$ and $\nu\in \Llog(\Omega_2)$, the tensor product $\pi=\mu\otimes\nu$ is a feasible candidate with finite energy by \cref{thm:sum}.
    Thus, the infimum in \eqref{eq:entropyRegularized} is finite, and weak duality -- which always holds due to the properties of supremum and infimum -- shows that the supremum in \eqref{eq:entropyRegularizedDual-formal} is finite as well. Existence of a solution for \eqref{eq:entropyRegularized} now follows from \cref{thm:primalExistence}.

    Uniqueness and regularity of the minimizer then are a direct consequence of the strict convexity of the entropy and \cref{thm:finiteEntropy}.
\end{proof}

In case a minimizer exists, we can characterize its support.
Here and throughout the rest of the paper, we use the usual shorthand $\{f>\lambda\}$ for the set $\{x\in \Omega:f(x)>\lambda\}$. We also recall from \cref{rem:orlicz} that $\1_A$ refers to the characteristic function of the set $A$. The following result can also be found in \cite[Thm.~2.7]{Borwein:1994}, but the proof there needs a constraint qualification for the primal problem which we do not need in this formulation. We present a full proof for the sake of completeness.
\begin{proposition}\label{thm:solutionSupport}
    A minimizer $\bar \pi\in \Llog(\Omega_1\times \Omega_2)$ of \eqref{eq:entropyRegularized} satisfies
    $\supp\bar\pi=\supp\mu\times\supp\nu$.
\end{proposition}
\begin{proof}
    The fact that $\supp\bar\pi\subset\supp\mu\times\supp\nu$ follows from the marginal constraints and the nonnegativity of $\bar\pi$.
    It remains to show that $\supp\bar\pi\supset\supp\mu\times\supp\nu$.
    For a contradiction, assume there is some $\hat x\in(\supp\mu\times\supp\nu)\setminus\supp\bar\pi$,
    then there exists a radius $r>0$ such that $\bar\pi=0$ on each ball $B_s(\hat x)$ with $s<r$, but $\mu(P_1(B_s(\hat x)))>0$ and $\nu(P_2(B_s(\hat x)))>0$.
    In particular, there exist $\omega_1\subset P_1(B_{r/2}(\hat x))$ and $\omega_2\subset P_2(B_{r/2}(\hat x))$ such that $\mu(\omega_i)>0$ and $\lebesgue(\omega_i)>0$ for $i=1,2$, but $\bar\pi(\omega_1\times\omega_2)=0$.
    We may choose $\omega_1,\omega_2$ small enough and $\varepsilon>0$ small enough such that there are $\tilde\omega_1\subset\Omega_1\setminus\omega_1$ and $\tilde\omega_2\subset\Omega_2\setminus\omega_2$ with nonzero Lebesgue measure and with $\bar\pi>\varepsilon$ on $(\tilde\omega_1\times\omega_2)\cup(\omega_1\times\tilde\omega_2)$.

    Let now
    $\kappa_i := \frac{\lebesgue(\omega_i)}{\lebesgue(\tilde{\omega}_i)}$ for $i=1,2$,
    $\kappa := \lebesgue(\omega_1)\cdot\lebesgue(\omega_2)$,
    and 
    \begin{equation*}
        \tilde{\pi} = \bar\pi + t\left[\1_{\omega_1\times\omega_2}(x_1,x_2) + \kappa_1 \kappa_2 \1_{\tilde{\omega}_1\times\tilde{\omega}_2}(x_1,x_2) - \kappa_1  \1_{\tilde{\omega}_1\times\omega_2}(x_1,x_2) - \kappa_2  \1_{\omega_1\times\tilde{\omega}_2}(x_1,x_2)\right]
    \end{equation*}
    for $0<t < \varepsilon/\min\{\kappa_1,\kappa_2\}$. Then $\tilde{\pi}$ is feasible. We will now argue that for small enough $t$ we have
    \begin{equation*}
        \int_M c \tilde{\pi}\bd (x_1,x_2) + \gamma\int_M \tilde{\pi}\log\tilde{\pi} \bd (x_1,x_2) \leq \int_M c \bar\pi\bd(x_1,x_2) + \gamma\int_M \bar\pi\log\bar\pi \bd (x_1,x_2),
    \end{equation*}
    where
    $M := (\omega_1 \cup \tilde\omega_1) \times (\omega_2\cup\tilde\omega_2)$.
    Note that $\bar\pi = \tilde{\pi}$ on $\Omega_1\times\Omega_2\setminus M$.

    First, consider $\int_M c \tilde{\pi} \bd (x_1,x_2)$. Since $c$ is continuous and finite,
    $\int_M c \tilde{\pi}\bd (x_1,x_2) - \int_M c \bar\pi\bd(x_1,x_2)$
    is finite and hence
    \begin{equation*}
        \int_M c \tilde{\pi} \bd (x_1,x_2) = \int_M c \bar\pi \bd(x_1,x_2) + tC_0
    \end{equation*}
    for some constant $C_0$.
    Now, consider the entropy of $\tilde{\pi}$. Since $\bar\pi = 0$ on $\omega_1 \times\omega_2$, we have
    \begin{equation*}
        \int_{\omega_1\times\omega_2} \tilde{\pi}\log\tilde{\pi}\bd (x_1,x_2) = \int_{\omega_1\times\omega_2} t\log t \bd (x_1,x_2)= \kappa t\log t.
    \end{equation*}
    Using the inequality $f(y) \geq f(x) + f'(x)(y-x)$ for convex and differentiable $f$, we can estimate
    \begin{align*}
        \int_{\omega_1\times\tilde{\omega}_2} \tilde{\pi}\log\tilde{\pi} \bd (x_1,x_2) &= \int_{\omega_1\times\tilde{\omega}_2} (\bar\pi - \kappa_2 t)\log(\bar\pi - \kappa_2 t)\bd (x_1,x_2)\\
        &\leq \int_{\omega_1\times\tilde{\omega}_2} \bar\pi\log\bar\pi \bd (x_1,x_2) - \kappa_2 t\int_{\omega_1\times\tilde{\omega}_2} \log(\bar\pi - \kappa_2 t) \bd (x_1,x_2) - \kappa t,
    \end{align*}
    and similarly for $\int_{\tilde{\omega}_1\times\omega_2} \tilde{\pi}\log\tilde{\pi} \bd (x_1,x_2)$.
    Again using the above inequality we have
    \begin{align*}
        \int_{\tilde{\omega}_1\times\tilde{\omega}_2} \tilde{\pi}\log\tilde{\pi}\bd (x_1,x_2)
        &\leq \int_{\tilde{\omega}_1\times\tilde{\omega}_2} \bar\pi\log\bar\pi\bd (x_1,x_2) + \kappa_1\kappa_2t\int_{\tilde{\omega}_1\times\tilde{\omega}_2} \log(\bar\pi+\kappa_1\kappa_2t)\bd (x_1,x_2) + \kappa t\\
        &\leq \int_{\tilde{\omega}_1\times\tilde{\omega}_2} \bar\pi\log\bar\pi\bd (x_1,x_2) + \kappa_1\kappa_2t\int_{\tilde{\omega}_1\times\tilde{\omega}_2} \bar\pi+t\kappa_1\kappa_2\bd (x_1,x_2).
    \end{align*}
    We obtain
    \begin{multline*}
        \int_M \tilde{\pi}\log\tilde{\pi} \bd (x_1,x_2)- \int_M \bar\pi\log\bar\pi \bd (x_1,x_2)
        \leq \kappa t\log t - \kappa_2 t\int_{\omega_1\times\tilde{\omega}_2}\log(\bar\pi-\kappa_2 t)\bd (x_1,x_2) - \kappa t\\
        - \kappa_1 t\int_{\tilde{\omega}_1\times\omega_2}\log(\bar\pi - \kappa_1 t)\bd (x_1,x_2) - \kappa t + \kappa_1\kappa_2t\int_{\tilde{\omega}_1\times\tilde{\omega}_2} \bar\pi+t\kappa_1\kappa_2 \bd (x_1,x_2).
    \end{multline*}
    The right-hand side is of the form  $g(t) = \kappa t\log t + h(t)$ with $h$ differentiable at $0$. We can therefore estimate
    \begin{equation*}
        g(t) \leq \kappa t\log t + C_1 t = t(\kappa \log t + C_1)
    \end{equation*}
    for some $C_1>0$ big enough and small $t$.

    Combining the estimates for cost and entropy yields
    \begin{multline*}
        \int_M c \tilde{\pi}\bd (x_1,x_2) + \gamma\int_M \tilde{\pi}\log\tilde{\pi} \bd (x_1,x_2)\\
        \leq \int_M c \bar\pi\bd(x_1,x_2) + \int_M \bar\pi\log\bar\pi \bd (x_1,x_2) + t(\gamma\kappa\log t + C_0 + \gamma C_1)
    \end{multline*}
    for $t$ small enough. However, the last term will be negative for $t$ small enough, which shows that $\bar \pi$ is not optimal in contradiction to the assumption.
\end{proof}

\bigskip

\Cref{thm:solutionProperties} shows that the natural setting for the entropically regularized problem \eqref{eq:entropyRegularized} is in fact $\Llog(\Omega)$ rather than $\meas(\Omega)$. In the next section, we will prove existence of solutions for a suitable modified dual problem of \eqref{eq:entropyRegularized} and justify a pointwise almost everywhere optimality system that can be used as a basis for deriving the Sinkhorn algorithm.

\section{Duality in \texorpdfstring{$\scriptstyle\Llog$}{LlogL} and \texorpdfstring{$\scriptstyle\Lexp$}{Lexp}}
\label{sec:duality-llogl-lexp}

In this section, we consider \eqref{eq:entropyRegularized} in the space $\Llog(\Omega_1\times\Omega_2)$. To derive a dual problem in $\Lexp(\Omega_1)\times \Lexp(\Omega_2)$, we shall perform the variable substitution 
\begin{equation*}
    \Phi(s)=
    \begin{cases}
        \infty&\text{if }s<0,\\
        s&\text{if }s\in[0,1],\\
        e^{s-1}&\text{else},
    \end{cases}
    \qquad\text{and}\qquad
    \Psi(s)=\log\Phi(s)=
    \begin{cases}
        -\infty&\text{if }s\leq 0,\\
        \log s&\text{if }s\in(0,1),\\
        {s-1}&\text{else},
    \end{cases}
\end{equation*}
see \cref{fig:PhiPsi}.
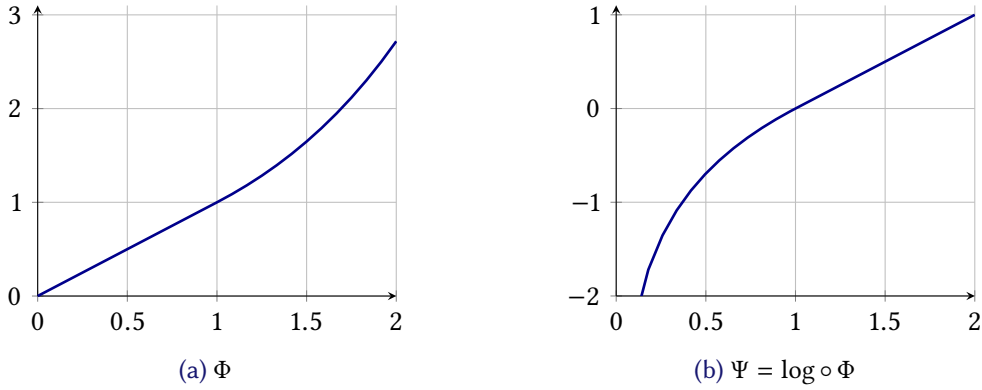
\begin{figure}[t]
    \centering
    \begin{subfigure}{0.4\textwidth}
        \centering
        \begin{tikzpicture}
            \begin{axis}[variable=\s,
                domain=0:2,
                xmin=0,
                xmax=2,
                ymin=0,
                ymax=3.1,
                axis lines=left,
                grid,
                width=\textwidth
                ]
                \addplot[
                    DarkBlue,
                    line width=1pt
                    ] {(s<=1)*s + (s>1)*exp(s-1)};
            \end{axis}
        \end{tikzpicture}
        \caption{$\Phi$}
    \end{subfigure}
    \hspace{1cm}
    \begin{subfigure}{0.4\textwidth}
        \centering
        \begin{tikzpicture}
            \begin{axis}[variable=\s,
                domain=0.1:2,
                xmin=0,
                xmax=2,
                ymin=-2,
                ymax=1.1,
                axis lines=left,
                grid,
                width=\textwidth
                ]
                \addplot[
                    DarkBlue,
                    line width=1pt
                    ] {and(s>=0.1,s<=1)*ln(s) + (s>1)*(s-1)};
            \end{axis}
        \end{tikzpicture}
        \caption{$\Psi=\log \circ\, \Phi$}
    \end{subfigure}
    \caption{Young functions for $\Llog$ and $\Lexp$}
    \label{fig:PhiPsi}
\end{figure}
Note that $\Phi$ is convex and $\Psi$ concave and that
the function $\Phi$ coincides with the Young function $\Phi_{\exp}$ from~\eqref{eq:Phiexp}, which is associated with $\Lexp$.

We now substitute $e^{\alpha/\gamma}=\Phi(u_1)$ and $e^{\beta/\gamma}=\Phi(u_2)$, i.e.,
\begin{equation}\label{eq:transformation_u}
    u_1=\begin{cases}
        e^{\alpha/\gamma}&\text{if }\alpha\leq0,\\
        \frac\alpha\gamma+1&\text{else},
    \end{cases}\qquad\qquad
    u_2=\begin{cases}
        e^{\beta/\gamma}&\text{if }\beta\leq0,\\\frac\beta\gamma+1&\text{else},
    \end{cases}
\end{equation}
which conversely implies that $\alpha = \gamma\log(\Phi(u_{1})) = \gamma\Psi(u_{1})$ and $\beta = \gamma\log(\Phi(u_{2})) = \gamma\Psi(u_{2})$.
Using this substitution, we obtain that
\begin{multline*}
    \int_{\Omega_{1}}\alpha(x_1)\bd\mu(x_1) + \int_{\Omega_{2}}\beta(x_2)\bd\nu(x_2) - \gamma\int_{\Omega_{1}\times\Omega_{2}}\exp\left(\tfrac{-c(x_1,x_2)+\alpha(x_1)+\beta(x_2)}\gamma\right)\bd(x_1,x_2) \\
    = -\gamma \int_{\Omega_1\times\Omega_2}\Phi(u_1(x_1))\Phi(u_2(x_2))e^{-\frac{c(x_1,x_2)}\gamma}\bd(x_1,x_2)+\gamma\int_{\Omega_1}\Psi(u_1)\mu\bd x_1+\gamma\int_{\Omega_2}\Psi(u_2)\nu\bd x_2.
\end{multline*}

Instead of the predual problem \eqref{eq:entropyRegularizedDual-formal}, we thus consider the reformulated problem
\begin{multline}\label{eq:entropyRegularizedDual-formal_u1_u2}
    \sup_{\substack{u_1\in \Lexp(\Omega_1)\\u_2\in \Lexp(\Omega_2)\\u_1,u_2\geq0}}
    -\gamma \int_{\Omega_1\times\Omega_2}\Phi(u_1(x_1))\Phi(u_2(x_2))e^{-\frac{c(x_1,x_2)}\gamma}\bd(x_1,x_2)\\
    +\gamma\int_{\Omega_1}\Psi(u_1)\mu\bd x_1+\gamma\int_{\Omega_2}\Psi(u_2)\nu\bd x_2\,.
    \tag{D$_{\exp}$}
\end{multline}
This substitution renders the problem nonconvex but, as we will see, allows to prove existence of solutions.

In the following, we assume that $\mu,\nu\in \Llog(\Omega)$ -- as required for existence for the primal problem -- and that $c\in\cont(\Omega_1\times\Omega_2)$. We also recall that the Luxemburg norms $\norm{\cdot}_{\Phi_{\exp}}$ and $\norm{\cdot}_{\Phi_{\log}}$ are equivalent norms on $\Lexp(\Omega)$ and $\Llog(\Omega)$, respectively. Our aim is to apply Tonelli's direct method to \eqref{eq:entropyRegularizedDual-formal_u1_u2} by showing that the functional
\begin{equation}\label{eq:dual-B}
    B(u_1,u_2) := \int_{\Omega_1\times\Omega_2}\Phi(u_1(x_1))\Phi(u_2(x_2))e^{-\frac{c(x_1,x_2)}\gamma}\bd(x_1,x_2)-\int_{\Omega_1}\Psi(u_1)\mu\bd x_1-\int_{\Omega_2}\Psi(u_2)\nu\bd x_2
\end{equation}
is radially unbounded and lower semi-continuous in the right topology. We first need the following lemma.
\begin{lemma}\label{thm:positiveExpEstimate}
    If $\norm{v\1_{\{v>0\}}}_{\Phi_{\exp}}>\max(1,\lebesgue(\Omega))$, then
    \begin{equation*}
        \norm{v\1_{\{v>0\}}}_{\Phi_{\exp}} \leq 
        \log\left(\frac1{\lebesgue(\Omega)}\int_\Omega\Phi((v+1)\1_{\{v>0\}})\bd x\right)/\log\frac e{\min(1,\lebesgue(\Omega))}\,.  
    \end{equation*}
\end{lemma}
\begin{proof}
    Set $\gamma_\varepsilon=\norm{v\1_{\{v>0\}}}_{\Phi_{\exp}}-\varepsilon$ for some $\varepsilon>0$ such that still $\gamma_{\epsilon}>\max(1,\lebesgue(\Omega))$.
    Then it holds that $\int_{\Omega}\Phi\big(\tfrac{v\1_{\{v>0\}}}{\gamma_{\epsilon}}\big)\bd x >\norm{v\1_{\{v>0\}}}_{\Phi_{\exp}} >\max(1,\lebesgue(\Omega))$.

    By Jensen's inequality we have
    \begin{equation*}
        \begin{aligned}
            \frac{\min(1,\lebesgue(\Omega))}{e}\left(\frac1{\lebesgue(\Omega)}\int_\Omega\Phi((v+1)\1_{\{v>0\}})\bd x\right)^{\frac1{\gamma_\varepsilon}}
            & =\frac{\min(1,\lebesgue(\Omega))}{e}\left(\frac1{\lebesgue(\Omega)}\int_\Omega \1_{\{v>0\}}e^v\bd x\right)^{\frac1{\gamma_\varepsilon}}\\
            &\geq\frac{\min(1,\tfrac1{\lebesgue(\Omega)})}{e}\int_\Omega \1_{\{v>0\}}\left(e^v\right)^{\frac1{\gamma_\varepsilon}}\bd x\\
            & = \min(1,\tfrac1{\lebesgue(\Omega)}) \int_{\Omega}\1_{\{v>0\}}e^{\tfrac{v}{\gamma_{\epsilon}}-1}\bd x\\
            &\geq \min(1,\tfrac1{\lebesgue(\Omega)})\int_\Omega\Phi\left(\tfrac{v\1_{\{v>0\}}}{\gamma_\varepsilon}\right)\bd x\\
            & >\min(1,\tfrac1{\lebesgue(\Omega)})\max(1,\lebesgue(\Omega)) = 1.
        \end{aligned}
    \end{equation*}
    Taking logarithms, we deduce that $\gamma_\varepsilon<\log\left(\frac1{\lebesgue(\Omega)}\int_\Omega\Phi((v+1)\1_{\{v>0\}})\bd x\right)/\log\frac e{\min(1,\lebesgue(\Omega))}$, and
    letting $\varepsilon\to 0$ yields the claim.
\end{proof}

We next capture the invariance inherited from \eqref{eq:entropyRegularizedDual-formal} as described in \cref{rem:existence}.

\begin{lemma}\label{lem:invariance}
    Let $u_i\in\Lexp(\Omega_i)$, $i=1,2$, with $B(u_1,u_2)<\infty$.
    If for an arbitrary $K\in\R$ we set $\tilde u_1=\Psi^{-1}(\Psi(u_1)-K)$ and $\tilde u_2=\Psi^{-1}(\Psi(u_1)+K)$,
    then $B(\tilde u_1,\tilde u_2)=B(u_1,u_2)$.
    In particular, by choosing $K$ appropriately, we can always achieve $\int_{\Omega_1}\Phi(\tilde u_1)\bd x_1= 1$.
\end{lemma}
\begin{proof}
    Note that $u_1 > 0$ $\mu$-a.e. and $u_2>0$ $\nu$-a.e. as $B(u_1,u_2) < \infty$. By construction, the same holds for $\tilde u_1$ and $\tilde u_2$.
    The first statement is now a direct consequence of the invariance of the cost functional in \eqref{eq:entropyRegularizedDual-formal} under the mapping $(\alpha,\beta)\mapsto(\alpha-K,\beta+K)$. For the second statement, first note that $\int_{\Omega_1}\Phi(\tilde u_1)\bd x_1$ is continuous in $K$. Moreover,
    \begin{align*}
        \int_{\Omega_1}\Phi(\tilde u_1)\bd x_1 &\to 0,\quad\text{for } K\to\infty\quad\text{and}\\
        \int_{\Omega_1}\Phi(\tilde u_1)\bd x_1 &\to \infty,\quad\text{for } K\to -\infty
    \end{align*}
    so that the assertion follows by the intermediate value theorem.
\end{proof}

\begin{remark}
    While $\|\tilde u_1\|_{\Phi_{\exp}}=1$ implies $\int_{\Omega_1}\Phi(\tilde u_1)\bd x_1\leq 1$ and $\int_{\Omega_1}\Phi(\tilde u_1)\bd x_1 = 1$ implies $\norm{\tilde u_1}_{\Phi_{\exp}}\leq 1$, in general we cannot achieve both equalities simultaneously due to \cref{rem:orlicz}.
\end{remark}

Modulo this invariance we now obtain coercivity.

\begin{lemma}\label{lem:coercive}
    Let $u_1^n$, $n=1,2,\ldots$, be a sequence in $\Lexp(\Omega_1)$ such that $\int_{\Omega_1}\Phi(u^n_1) = 1$ for all $n$.
    Then $\norm{u_2^n}_{\Phi_{\exp}}\to \infty$ for $n\to\infty$ implies $B(u_1^n,u_2^n)\to \infty$ as $n\to\infty$.
\end{lemma}
\begin{proof}
    Without loss of generality we may assume the $u_2^n$ to be nonnegative, since replacing $u_2^n$ with its absolute value decreases $B(u_1^n,u_2^n)$ without changing $\norm{u_2^n}_{\Phi_{\exp}}$.
    Due to $\int_{\Omega_1}\Phi(u_1^n)\bd x_1 = 1$ we have  $\norm{u_1^n}_{\Phi_{\exp}}\leq 1$ and thus
    \begin{equation*}
        \int_{\Omega_1}\Psi\left(u_1^n\right)\mu\bd x_1\leq\int_{\Omega_1}(u_1^n-1)\mu\bd x_1\leq c_\Phi\norm{u_1^n}_{\Phi_{\exp}}\norm{\mu}_{\Phi_{\log}}-1\leq c_\Phi\norm{\mu}_{\Phi_{\log}}-1,
    \end{equation*}
    where $c_\Phi$ denotes the generic equivalence constant for the duality from \cref{thm:dual-LlogL}.
    Analogously we obtain
    \begin{equation*}
        \int_{\Omega_2}\Psi\left(u_2^n\right)\nu\bd x_2\leq\int_{\Omega_2}\max(u_2^n-1,0)\nu\bd x_2\leq c_\Phi\norm{\max(u_2^n-1,0)}_{\Phi_{\exp}}\norm{\nu}_{\Phi_{\log}}.
    \end{equation*}
    Hence for $C=\exp(-\max_{\Omega_1\times\Omega_2}c/\gamma)$, we have
    \begin{equation*}
        \begin{aligned}
            B(u_1^n,u_2^n)
            &\geq C\int_{\Omega_1}\Phi(u_1^n)\bd x_1\int_{\Omega_2}\Phi(u_2^n)\bd x_2-\int_{\Omega_1}\Psi\left(u_1^n\right)\mu\bd x_1-\int_{\Omega_2}\Psi\left(u_2^n\right)\nu\bd x_2\\
            &\geq C\int_{\Omega_2}\Phi(u_2^n)\bd x_2-c_\Phi\norm{\mu}_{\Phi_{\log}}+1-c_\Phi\norm{\max(u_2^n-1,0)}_{\Phi_{\exp}}\norm{\nu}_{\Phi_{\log}}.
        \end{aligned}
    \end{equation*}
    Since $\norm{u_2^n}_{\Phi_{\exp}}\to\infty$ and $u_2^n$ is nonnegative, we also have $\norm{\max(u_2^n-1,0)}_{\Phi_{\exp}}\to\infty$ as $n\to\infty$ and therefore
    \begin{equation*}
        \norm{\max(u_2^n-1,0)}_{\Phi_{\exp}}\leq\log\left(\frac1{\lebesgue(\Omega_2)}\int_{\Omega_2}\Phi(u_2^n)\bd x_2\right)/\log\left(\frac e{\min(1,\lebesgue(\Omega_2))}\right)\quad\text{ for $n$ large enough}
    \end{equation*}
    by \cref{thm:positiveExpEstimate}. 
    Now \cref{thm:OrliczFunctionEstimate} implies that $\int_{\Omega_{2}}\Phi(u^n_{2})\bd x\to\infty$ for $n\to \infty$ and therefore that
    \begin{equation*}
        B(u_1^n,u_2^n)
        \geq C\int_{\Omega_{2}}\Phi(u^n_{2})\bd x - c_\Phi\norm{\mu}_{\Phi_{\log}}+1-c_\Phi\frac{\log\left(\frac1{\lebesgue(\Omega_2)}\int_{\Omega_2}\Phi(u_2^n)\bd x_2\right)}{\log\Big(\tfrac e{\min(1,\lebesgue(\Omega_2))}\Big)}\norm{\nu}_{\Phi_{\log}}
        \to\infty,
    \end{equation*}
    which yields the desired contradiction.
\end{proof}

\begin{lemma}\label{lem:wlsc}
    $B$ is sequentially weakly-$*$ lower semi-continuous on $\Lexp(\Omega_1)\times \Lexp(\Omega_2)$.
\end{lemma}
\begin{proof}
    Let $(u_1^n,u_2^n)\weakstarto(u_1,u_2)$ in $\Lexp(\Omega_1)\times \Lexp(\Omega_2)$.
    Then we have in particular $(u_1^n,u_2^n)\rightharpoonup(u_1,u_2)$ in $L^p(\Omega_1)\times L^p(\Omega_2)$ for any $1\leq p<\infty$. Since $-\Psi$ is a lower semi-continuous and convex integrand, it thus follows, e.g., by \cite[Thm.~13.1.1]{Attouch:2006} that
    \begin{equation*}
        \liminf_{n\to\infty} \int_{\Omega_1}-\Psi(u_1^n)\mu\bd x_1 \geq \int_{\Omega_1}-\Psi(u_1)\mu\bd x_1,
        \qquad
        \liminf_{n\to\infty} \int_{\Omega_2}-\Psi(u_2^n)\mu\bd x_1 \geq \int_{\Omega_2}-\Psi(u_2)\mu\bd x_2,
    \end{equation*}
    and hence that these functionals are weak-$*$ sequentially lower semicontinuous on $\Lexp(\Omega_1)\times \Lexp(\Omega_2)$.

    It remains to show weak-$*$ lower semi-continuity of $\int_{\Omega_1\times\Omega_2}\Phi(u_1(x_1))\Phi(u_2(x_2))e^{-\frac{c(x_1,x_2)}\gamma}\bd(x_1,x_2)$.
    For fixed $N>0$, decompose $\Omega_1$ and $\Omega_2$ into a finite number of subsets $\Omega_i^k$ with $\lebesgue(\Omega_i^k)\leq\frac1N$.
    We further assume that the decompositions $(\Omega_1^k)_k$ and $(\Omega_2^k)_k$ for $N+1$ are obtained from the decompositions for $N$ by refinement. 
    Defining $c_{kl}=\min_{(x_1,x_2)\in \Omega_1^k\times \Omega_2^l} e^{-c(x_1,x_2)/\gamma}$, we then have
    \begin{multline*}
        \liminf_{n\to\infty}\int_{\Omega_1\times\Omega_2}\Phi(u_1^n(x_1))\Phi(u_2^n(x_2))e^{-\frac{c(x_1,x_2)}\gamma}\bd(x_1,x_2)\\
        \begin{aligned}
            &\geq\liminf_{n\to\infty}\int_{\Omega_1\times\Omega_2}\Phi(u_1^n(x_1))\Phi(u_2^n(x_2))\sum_{k,l}c_{kl}\1_{\Omega_1^k\times\Omega_2^l}\bd(x_1,x_2)\\
            &=\liminf_{n\to\infty}\sum_{k,l}c_{kl}\int_{\Omega_1\times\Omega_2}\Phi(u_1^n(x_1))\Phi(u_2^n(x_2))\1_{\Omega_1^k}(x_1)\1_{\Omega_2^l}(x_2)\bd(x_1,x_2)\\
            &\geq\sum_{k,l}c_{kl}\liminf_{n\to\infty}\int_{\Omega_1}\Phi(u_1^n(x_1))\1_{\Omega_1^k}(x_1)\bd x_1\,\liminf_{n\to\infty}\int_{\Omega_2}\Phi(u_2^n(x_2))\1_{\Omega_2^l}(x_2)\bd x_2\,.
        \end{aligned}
    \end{multline*}
    Similarly as above, it follows from the lower semi-continuity and convexity of $\Phi$ that $u\mapsto\int_{\Omega_1^k}\Phi(u)\bd x_1$ and $v\mapsto\int_{\Omega_2^l}\Phi(v)\bd x_2$ are sequentially weakly-$*$ lower semi-continuous on $\Lexp(\Omega_1^k)$ and $\Lexp(\Omega_2^l)$, respectively. Hence
    \begin{multline*}
        \sum_{k,l}c_{kl}\liminf_{n\to\infty}\int_{\Omega_1}\Phi(u_1^n(x_1))\1_{\Omega_1^k}(x_1)\bd x_1\,\liminf_{n\to\infty}\int_{\Omega_2}\Phi(u_2^n(x_2))\1_{\Omega_2^l}(x_2)\bd x_2\\
        \begin{aligned}
            &\geq\sum_{k,l}c_{kl}\int_{\Omega_1}\Phi(u_1(x_1))\1_{\Omega_1^k}(x_1)\bd x_1\,\int_{\Omega_2}\Phi(u_2(x_2))\1_{\Omega_2^l}(x_2)\bd x_2\\
            &=\int_{\Omega_1\times\Omega_2}\Phi(u_1(x_1))\Phi(u_2(x_2))\sum_{k,l}c_{kl}\1_{\Omega_1^k\times\Omega_2^l}(x_1,x_2)\bd(x_1,x_2)\\
            &\to\int_{\Omega_1\times\Omega_2}\Phi(u_1(x_1))\Phi(u_2(x_2))e^{-\frac{c(x_1,x_2)}\gamma}\bd(x_1,x_2)\quad\text{as }N\to\infty
        \end{aligned}
    \end{multline*}
    by the monotone convergence theorem, since $\sum_{k,l}c_{kl}\1_{\Omega_1^k\times\Omega_2^l}\nearrow e^{-c/\gamma}$ monotonically.
\end{proof}

\begin{theorem}[dual existence]\label{thm:dualExistence}
    Problem \eqref{eq:entropyRegularizedDual-formal_u1_u2} possesses a maximizer $(\bar u_{1},\bar u_{2})\in \Lexp(\Omega_1)\times \Lexp(\Omega_2)$.
\end{theorem}
\begin{proof}
    We show that $B$ possesses a minimizer.
    The energy $B$ is finite at, e.g., $u_1\equiv 1\equiv u_2$.
    We thus may consider a minimizing sequence $(u_1^n,u_2^n)$ in $\Lexp(\Omega_1)\times\Lexp(\Omega_2)$,
    where by \cref{lem:invariance} we may assume $\int_{\Omega_1}\Phi_{\exp}(u^n_1)\bd x_1 = 1$ without loss of generality.
    \Cref{lem:coercive} now implies boundedness of $\norm{u_2^n}_{\Phi_{\exp}}$
    so that by the Banach--Alaoglu theorem we may extract a weakly-$*$ convergent subsequence from $(u_1^n,u_2^n)$
    (recalling that $\Llog(\Omega_1\times\Omega_2)$ is separable by \cref{lem:llogl-separable}).
    The claim now follows from the lower semi-continuity of $B$ along that subsequence by \cref{lem:wlsc}.
\end{proof}

From dual solutions $\bar u_1$ and $\bar u_2$, we obtain by backsubstitution $\bar \alpha := \gamma\Psi(\bar u_{1})$ and $\bar \beta := \gamma\Psi(\bar u_{2})$ as a candidate for a solution of the original predual problem \eqref{eq:entropyRegularizedDual-formal}. However, these are in general not admissible since $\bar u_{1}\in \Lexp(\Omega_1)$ and $\bar u_{2} \in \Lexp(\Omega_2)$ does not imply the needed regularity of $\bar \alpha\in \cont(\Omega_1)$ and $\bar \beta\in \cont (\Omega_2)$: The positive parts of $\bar \alpha$ and $\bar \beta$ (which equal the positive parts of $\bar u_1+1$ and $\bar u_2+1$, respectively) are in $\Lexp$,
but the negative parts need not even be functions as they could be $-\infty$ everywhere.

Nevertheless, from \eqref{eq:entropyRegularizedDual-formal_u1_u2} one sees that $\bar u_1\geq 0$ $\mu$-almost everywhere and $\bar u_2\geq 0$ $\nu$-almost everywhere, and hence $\bar u_1$ and $\bar u_2$ are at least $\mu$- and $\nu$-measurable, respectively.
We will derive more information on $\bar \alpha$ and $\bar \beta$ from the necessary optimality conditions. 

First, we have again a strong duality result relating \eqref{eq:entropyRegularizedDual-formal_u1_u2} to \eqref{eq:entropyRegularized}.
\begin{proposition}[strong duality]\label{thm:strong-duality}
    Let $\mu\in \Llog(\Omega_1)$, $\nu\in \Llog(\Omega_2)$, and $c\in\cont(\Omega_1\times\Omega_2)$. Then, both \eqref{eq:entropyRegularized} and \eqref{eq:entropyRegularizedDual-formal_u1_u2} admit a solution, and their optimal values coincide.
\end{proposition}
\begin{proof}
    Existence for both problems follows from \cref{thm:solutionProperties,thm:dualExistence}.
    To show their equality, by \cref{thm:primalExistence} it suffices to show that the value of \eqref{eq:entropyRegularizedDual-formal} equals that of \eqref{eq:entropyRegularizedDual-formal_u1_u2}.
    First, let $\alpha\in\cont(\Omega_1)$ and $\beta \in\cont(\Omega_2)$ be arbitrary and set $u_1:=\Psi^{-1}(\alpha/\gamma)$ and $u_2:=\Psi^{-1}(\beta/\gamma)$. By substitution, we see that
    \begin{multline*}
        \int_{\Omega_1}\alpha\mu\bd(x_1) + \int_{\Omega_2} \beta\nu\bd(x_2) - \gamma \int_{\Omega_1\times \Omega_2} \exp\left(\frac{-c(x_1,x_2)+\alpha(x_2)+\beta(x_1)}{\gamma}\right) \bd(x_1,x_2)\\
        \leq  \max_{u_1,u_2\geq 0}  -\gamma B(u_1,u_2),
    \end{multline*}
    and taking the supremum over all $\alpha,\beta$ yields that the value of \eqref{eq:entropyRegularizedDual-formal} is at most that of \eqref{eq:entropyRegularizedDual-formal_u1_u2}.

    It thus remains to show that the value of \eqref{eq:entropyRegularizedDual-formal_u1_u2} can be achieved by \eqref{eq:entropyRegularizedDual-formal}.
    Let $\bar u_1,\bar u_2$ be optimal. By the monotone convergence theorem,
    $B(\bar u_1,\bar u_2)=\lim_{n\to\infty}B(\max\{\bar u_1,\frac1n\},\max\{\bar u_1,\frac1n\})$
    and also 
    \begin{equation*}
        B\left(\max\left\{\bar u_1,\tfrac1n\right\},\max\left\{\bar u_1,\tfrac1n\right\}\right)=\lim_{N\to\infty}B\left(\min\left\{\max\left\{\bar u_1,\tfrac1n\right\},N\right\},\min\left\{\max\left\{\bar u_1,\tfrac1n\right\},N\right\}\right).
    \end{equation*}
    Hence $B(\bar u_1,\bar u_2)$ can be arbitrarily well approximated by $B(u_1,u_2)$ with $\Psi(u_1)\in L^\infty(\Omega_1)$ and $\Psi(u_2)\in L^\infty(\Omega_2)$.
    Now let $\alpha_n\in\cont(\Omega_1)$ and $\beta_n\in\cont(\Omega_2)$ with $\alpha_n\to\gamma\Psi(u_1)$ in $L^2(\Omega_1)$ and $\beta_n\to\gamma\Psi(u_2)$ in $L^2(\Omega_2)$.
    Here we may assume $\alpha_n,\beta_n$ to be uniformly bounded
    so that (upon restricting to a subsequence) we additionally have $\alpha_n\weakstarto\gamma\Psi(u_1)$ and $\beta_n\weakstarto\gamma\Psi(u_2)$ in $L^\infty(\Omega)$.
    Now 
    \begin{equation*}
        \int_{\Omega_1}\alpha_n\mu\bd x_1+\int_{\Omega_2}\beta_n\nu\bd x_2\to \gamma\left[\int_{\Omega_1}\Psi(u_1)\mu\bd x_1+\int_{\Omega_2}\Psi(u_2)\nu\bd x_2\right]
    \end{equation*}
    due to the weak-$*$ convergence.
    Finally, as $\alpha_n, \beta_n$ converge in $L^2$, $e^{\frac{\alpha_n(x_1)+\beta_n(x_2)}\gamma}$ converges a.e. (after passing to a subsequence). Using uniform boundedness of $\alpha_n, \beta_n$, the dominated convergence theorem yields
    \begin{equation*}
        \int_{\Omega_1\times\Omega_2}\exp\left(\frac{\alpha_n(x_1)+\beta_n(x_2)-c(x_1,x_2)}\gamma\right)\bd(x_1,x_2)\to\int_{\Omega_1\times\Omega_2}\Phi(u_1(x_1))\Phi(u_2(x_2))e^{-\frac{c(x_1,x_2)}\gamma}\bd(x_1,x_2)\,.
        \qedhere
    \end{equation*}
\end{proof}

Having established primal and dual existence, we can now show how the solution of the dual problem can be used to solve the primal problem.
\begin{theorem}[optimality conditions]
    Let $\mu\in \Llog(\Omega_1)$, $\nu\in \Llog(\Omega_2)$, and $c\in\cont(\Omega_1\times\Omega_2)$.    
    Then  solutions   $(\bar u_1,\bar u_2)\in \Lexp(\Omega_1)\times\Lexp(\Omega_2)$ of \eqref{eq:entropyRegularizedDual-formal_u1_u2} satisfy
    \begin{subequations}
        \label{eq:optCond}
        \begin{align}
            \int_{\Omega_2}\Phi(\bar u_2(x_2))e^{-\frac{c(x_1,x_2)}\gamma}\bd x_2\,\Phi(\bar u_1(x_1)) & = \mu(x_{1})\,,\label{eq:optCond1}\\
            \int_{\Omega_1}\Phi(\bar u_1(x_1))e^{-\frac{c(x_1,x_2)}\gamma}\bd x_1\,\Phi(\bar u_2(x_2)) & = \nu(x_{2})\,,\label{eq:optCond2}
        \end{align}
    \end{subequations}
    for $\mu$-almost every $x_1\in\Omega_1$ and $\nu$-almost every $x_2\in \Omega_2$. Furthermore, $\bar\pi$ defined by
    \begin{equation}
        \bar \pi(x_1,x_2)=\Phi(\bar u_1(x_1))\Phi(\bar u_2(x_2))e^{-\frac{c(x_1,x_2)}\gamma}\label{eq:optimumPi}
    \end{equation}
    is the solution of \eqref{eq:entropyRegularized}.
\end{theorem}
\begin{proof}
    Let $\bar u_1,\bar u_2$ be solutions of the dual problem.
    We start with deriving the necessary conditions \eqref{eq:optCond}.
    First, note that $\{\bar u_1>0\}\supset\{\mu>0\}$ and $\{\bar u_2>0\}\supset\{\nu>0\}$ (up to a Lebesgue-negligible set)
    since otherwise $\int_{\Omega_1}\Psi(\bar u_1)\mu\bd x_1+\int_{\Omega_2}\Psi(\bar u_2)\nu\bd x_2=-\infty$. 
    Let now $\varepsilon>0$ be arbitrary and consider any $\varphi\in \Lexp(\Omega_1)\cap L^\infty(\Omega_1)$ with $\varphi=0$ on $\{\bar u_1<\varepsilon\}$. We next argue that the dual functional $B$ given in \eqref{eq:dual-B} is directionally differentiable in $(\bar u_1,\bar u_2)$ with respect to its first argument in direction $\varphi$. 
    Since both $s\mapsto\Phi(s)$ and $s\mapsto\Psi(s)$ are differentiable at $s>0$, so are the integrands pointwise almost everywhere on $\{\bar u_1 \geq \varepsilon\}\times\Omega_2$. It therefore suffices to show that the pointwise directional derivatives are integrable in order to differentiate under the integral.
    For the first term in $B$, we have almost everywhere on $\{\bar u_1 \geq \varepsilon\}\times\Omega_2$ that
    \begin{align*}
        \Phi'(\bar u_1;\varphi)\Phi(\bar u_2)e^{-\frac c\gamma} &=
        \begin{cases}
            \Phi(\bar u_1)\varphi \Phi(\bar u_2) e^{-\frac c\gamma}&\bar u_1>1\,,\\
            \varphi\Phi(\bar u_2)e^{-\frac c\gamma}&\text{else,}
        \end{cases}\\
        &\leq\frac1\varepsilon\|\varphi \|_\infty \Phi(\bar u_1)\Phi(\bar u_2) e^{-\frac c\gamma}\,,
    \end{align*}
    which is integrable on $\{\bar u_1 \geq \varepsilon\}\times \Omega_2$ since $\bar u_1$ and $\bar u_2$ are feasible for \eqref{eq:entropyRegularizedDual-formal_u1_u2}. An integrable lower bound is obtained similarly using $\varphi\in L^\infty(\Omega_1)$.

    For the second term in $B$, the chain rule and differentiability of $\Phi$ yields almost everywhere on $\{\bar u_1 \geq \varepsilon\}\times\Omega_2$ that
    \begin{equation*}
        \Psi'(\bar u_1 ;\varphi) = \frac{1}{\Phi(\bar u_1)}\Phi'(\bar u_1)\varphi = (\bar u_1^{-1}\1_{\{\bar u_1<1\}} + \1_{\{\bar u_1\geq1\}}) \varphi\leq \|\varphi\|_\infty(\varepsilon^{-1}\1_{\{\bar u_1<1\}} + \1_{\{\bar u_1\geq1\}})\,,
    \end{equation*}
    where the right-hand side is integrable with respect to $\mathrm{d}\mu$.

    From the dominated convergence theorem, it thus follows that the partial directional derivative of $B$ in the first direction is given by
    \begin{equation*}
        \begin{aligned}
            B_1'(\bar u_1,\bar u_2;\varphi)
            &=-\int_{\{\bar u_1\geq\varepsilon\}}\Psi'(\bar u_1;\varphi)\mu\bd x_1\\
            \MoveEqLeft[-1]+\int_{\{\bar u_1\geq\varepsilon\}\times\Omega_2}\Phi'(\bar u_1(x_1);\varphi(x_1))\Phi(\bar u_2(x_2))e^{-\frac{c(x_1,x_2)}\gamma}\bd(x_1,x_2)\\
            &=-\int_{\{\bar u_1\geq\varepsilon\}}\max\{\Phi(\bar u_1(x_1)),1\} \varphi(x_1)\left[\tfrac{\mu(x_1)}{\Phi(\bar u_1(x_1))}-\int_{\Omega_2}\Phi(\bar u_2(x_2))e^{-\frac{c(x_1,x_2)}\gamma}\bd x_2\right]\bd x_1\,,
        \end{aligned}
    \end{equation*}
    where we have again used the integrability of the integrand to apply Fubini's Theorem in order to iterate the double integrals and used $\Phi'(s) = \max\{\Phi(s),1\}$ and $\Psi'(s) = \Phi(s)^{-1} \Phi'(s)$.

    By the specific choice of $\varphi$, we have $\bar u_1 \pm t\varphi \geq 0$ for all $t>0$ sufficiently small. The optimality of $(\bar u_1,\bar u_2)$ thus implies that
    \begin{equation*}
        0 = B_1'(\bar u_1,\bar u_2; \varphi),
    \end{equation*}
    and since $\varphi$ was arbitrary on $\{\bar u_1\geq \varepsilon\}$ and $\max\{\Phi(\bar u_1),1\}>0$, we must therefore have that
    \begin{equation*}
        0=\mu(x_1)-\Phi(\bar u_1(x_1))\int_{\Omega_2}\Phi(\bar u_2(x_2))e^{-\frac{c(x_1,x_2)}\gamma}\bd x_2
        \qquad\text{for $\mu$-almost all } x_1\in \Omega_1\text{ with }\bar u_1(x_1)\geq\varepsilon\,.
    \end{equation*}
    Furthermore, since $\varepsilon>0$ was arbitrary and $\mu(x_1)=0$ whenever $\bar u_1(x_1)=0$, this equation even holds for $\mu$-almost all $x_1\in\Omega_1$,
    which yields \eqref{eq:optCond1}.
    Equation \eqref{eq:optCond2} is derived analogously.

    Now we show that $\bar \pi$ defined by \eqref{eq:optimumPi} is a solution of the primal problem.
    First note that by construction, $\bar\pi$ is feasible (i.e., is non-negative and has the correct marginals).
    Since strong duality holds by \cref{thm:strong-duality}, it thus suffices to show that the primal objective functional evaluated in $\bar \pi$ is equal to the dual optimal objective value \eqref{eq:entropyRegularizedDual-formal_u1_u2}.
    To that end, we insert \eqref{eq:optimumPi} into the objective functional in \eqref{eq:entropyRegularized} and obtain (using again the convention that $0\log 0 =0$)
    \begin{equation*}
        \begin{aligned}
            \int_{\Omega_1\times\Omega_2} c\bar\pi + \gamma\bar\pi(\log\bar\pi -1)\bd(x_1,x_2)
            &= \int_{\Omega_1\times\Omega_2} c \Phi(\bar u_1)\Phi(\bar u_2) e^{-\frac c\gamma}\bd(x_1,x_2)\\
            \MoveEqLeft[-3]+ \gamma\Phi(\bar u_1)\Phi(\bar u_2)e^{-\frac c\gamma}\left(\Psi(\bar u_1) + \Psi(\bar u_2) - \frac c\gamma -1\right)\bd(x_1,x_2)\\
            &= -\ \gamma\int_{\Omega_1\times\Omega_2} \Phi(\bar u_1)\Phi(\bar u_2)e^{-\frac c\gamma}\bd(x_1,x_2)\\
            \MoveEqLeft[-1]+ \gamma \int_{\Omega_1\times\Omega_2}\Psi(\bar u_1)\Phi(\bar u_1)\Phi(\bar u_2)e^{-\frac c\gamma}\bd(x_1,x_2)\\
            \MoveEqLeft[-1]+ \gamma\int_{\Omega_{1}\times\Omega_{2}} \Psi(\bar u_2)\Phi(\bar u_1)\Phi(\bar u_2)e^{-\frac c\gamma}\bd(x_1,x_2)\,.
        \end{aligned}
    \end{equation*}
    Since $\bar u_i \geq 0$ and hence $\Phi(\bar u_i)\geq 0$, we have that $\Psi(\bar u_i)\Phi(\bar u_i)= \log(\Phi(\bar u_i)) \Phi(\bar u_i) \geq -\frac1e$. Furthermore, we have assumed $\lebesgue(\Omega_i)<\infty$ and can thus shift the integrand to allow applying Tonelli's Theorem in the second and third integral. Inserting \eqref{eq:optCond}, the right-hand side now coincides with $B(\bar u_1,\bar u_2)$. Hence strong duality holds for $(\bar u_1,\bar u_2)$ and $\bar\pi$, and thus the latter is a solution to \eqref{eq:entropyRegularized}.
\end{proof}

\begin{remark}\label{rem:sinkhorn}
    The optimality system \eqref{eq:optCond} can be used to derive the Sinkhorn algorithm. First, note that one only needs to find $\bar u_{1}$ and $\bar u_{2}$ that solve \eqref{eq:optCond1} and \eqref{eq:optCond2}; an optimal plan $\bar \pi$ is then obtained from \eqref{eq:optimumPi}.
    The Sinkhorn method now solves the nonlinear system \eqref{eq:optCond} by alternatingly solving the equations: Given $u_{2}^{n}$, compute $u_{1}^{n+1}$ by solving \eqref{eq:optCond1}, i.e., setting
    \begin{equation*}
        u_{1}^{n+1}(x_{1}) = \Phi^{-1}\left(\frac{\mu(x_{1})}{\int_{\Omega_{2}}\Phi(u_{2}^{n}(x_{2}))\exp\left(-\tfrac{c(x_{1},x_{2})}{\gamma}\right)\bd x_{2}}\right),
    \end{equation*}
    and then solve \eqref{eq:optCond2} with $u_{2}^{n+1}$ to obtain
    \begin{equation*}
        u_{2}^{n+1}(x_{2}) = \Phi^{-1}\left(\frac{\nu(x_{2})}{\int_{\Omega_{1}}\Phi(u_{1}^{n+1}(x_{1}))\exp\left(-\tfrac{c(x_{1},x_{2})}{\gamma}\right)\bd x_{1}}\right).
    \end{equation*}
    Formulating this iteration directly in $\Phi(u_{1})$ and $\Phi(u_{2})$, we obtain the original Sinkhorn algorithm, cf.~\cite[Sec.~5.3.1]{Essid:2019}.
\end{remark}

\begin{remark}
    The optimality system \eqref{eq:optCond} also corresponds to the so-called Schrödinger system \cite[Eq. (4.12)--(4.13) or (4.14)]{Chen:2020}, i.e.,
    the system of equations which characterizes the solution to the so-called Schrödinger bridge problem
    (essentially, the most likely transition path of a hot gas between the initial and final gas distribution $\mu$ and $\nu$).
    Existence of solutions to that system was typically shown based on iterative approximation schemes
    (analogous to but predating the Sinkhorn algorithm; see the discussion in \cite{Chen:2020}).
    There are also alternative proofs  exploiting the variational nature of the problem; however, these are not as straightforward as identifying \eqref{eq:optCond} as the optimality conditions to an optimization problem which has a solution.
    In \cite{Beurling:1960}, for example, a minimizing sequence for the dual problem \eqref{eq:entropyRegularizedDual-formal-intro} is used to construct
    a sequence of measures of the type \eqref{eq:optimumPi} that is then shown to converge to a solution of the Schrödinger bridge problem.
\end{remark}

Finally, the optimality conditions \eqref{eq:optCond1} and \eqref{eq:optCond2} allow us to conclude which problem is solved by $(\bar \alpha,\bar \beta)$.
\begin{corollary}\label{cor:predual_lexp}
    Let $\mu\in \Llog(\Omega_1)$, $\nu\in \Llog(\Omega_2)$, and $c\in\cont(\Omega_1\times\Omega_2)$.    
    Let $(\bar u_1,\bar u_2)\in \Lexp(\Omega_1)\times\Lexp(\Omega_2)$ be a solution of \eqref{eq:entropyRegularizedDual-formal_u1_u2}. Then $\bar \alpha := \gamma\Psi(\bar u_{1})\in L^1(\Omega_1,\mu)$ and $\bar \beta := \gamma\Psi(\bar u_{2})\in L^1(\Omega_2,\nu)$ are solutions of
    \begin{equation}\label{eq:entropyRegularizedDual-l1}
        \sup_{\substack{\alpha\in L^1(\Omega_{1},\mu)\\\beta\in L^1(\Omega_{2},\nu)}}\int_{\Omega_{1}}\alpha\bd\mu+\int_{\Omega_{2}}\beta\bd\nu-\gamma\int_{\Omega_1\times\Omega_2}\exp\left(\tfrac{-c(x_1,x_2)+\alpha(x_2)+\beta(x_1)}\gamma\right)\bd(x_1,x_2),
        \tag{D$_{L^1}$}
    \end{equation}
    and the values of \eqref{eq:entropyRegularizedDual-l1} and \eqref{eq:entropyRegularizedDual-formal_u1_u2} coincide.
\end{corollary}
\begin{proof}
    First, note that the mapping $x_1\mapsto\int_{\Omega_2}\Phi(\bar u_2(x_2))e^{-\frac{c(x_1,x_2)}\gamma}\bd x_2$ is continuous and thus attains a minimum $\underline{c}>0$ and a maximum $\overline{c}>0$ on the (assumed to be) compact set $\Omega_1$.
    From the optimality condition \eqref{eq:optCond1}, we thus obtain that
    \begin{equation*}
        \underline{c}e^{\bar\alpha/\gamma}
        =\underline{c}\Phi(\bar u_1)
        \leq\mu
        \leq \overline{c}\Phi(\bar u_1)
        =\overline{c}e^{\bar\alpha/\gamma}\,.
    \end{equation*}
    This implies that $\log\mu-K\leq\bar\alpha/\gamma\leq\log\mu+K$ for some $K>0$.
    We thus have
    \begin{equation*}
        \tfrac1\gamma\int_{\Omega_1}|\bar\alpha|\mu\bd x_1
        \leq K\int_{\Omega_1}\mu\bd x_1+\int_{\Omega_1}|\log\mu|\mu\bd x_1.
    \end{equation*}
    Since $\mu\in \Llog(\Omega_1)$, we deduce that the right-hand side is finite and hence that $\bar\alpha$ is integrable with respect to $\mu$, i.e., $\bar\alpha\in L^{1}(\Omega_{1},\mu)$.
    The result for $\bar\beta$ follows analogously.
    Finally, it follows from a density argument that \eqref{eq:entropyRegularizedDual-l1} cannot exceed \eqref{eq:entropyRegularizedDual-formal_u1_u2}.
    Indeed, assume there are $\alpha\in L^1(\Omega_1,\mu)$ and $\beta\in L^1(\Omega_2,\nu)$ with an objective functional value $C$ strictly larger than \eqref{eq:entropyRegularizedDual-formal_u1_u2}.
    By invoking the monotone convergence theorem as in the proof of \cref{thm:strong-duality}, we may assume without loss of generality that $\alpha$ and $\beta$ are bounded.
    Defining now $u_1=\Psi^{-1}(\frac\alpha\gamma)\in L^\infty(\Omega_1)\subset\Lexp(\Omega_1)$ and $u_2=\Psi^{-1}(\frac\beta\gamma)\in L^\infty(\Omega_2)\subset\Lexp(\Omega_2)$
    shows that \eqref{eq:entropyRegularizedDual-formal_u1_u2} is no smaller than $C$, the desired contradiction.
\end{proof}

\begin{remark}\label{rem:uniqueness_dual}
    As for \eqref{eq:entropyRegularizedDual-formal} and as formalized in \cref{lem:invariance}, solutions to \eqref{eq:entropyRegularizedDual-formal_u1_u2} are not unique. 
\end{remark}

\section{\texorpdfstring{$\scriptstyle\Gamma$}{Γ}-limit}
\label{sec:gamma}

We now turn to $\Gamma$-convergence of the regularized problem.
Recall from, e.g., \cite{Braides:2002}, that a sequence $\{F_n\}$ of functionals $F_n:X\to \overline \R$ on a metric space $X$ is said to $\Gamma$-converge to a functional $F:X\to\overline \R$, written $F = \gammalim_{n\to\infty} F_n$, if
\begin{enumerate}[(i)]
    \item for every sequence $\{x_n\}\subset X$ with $x_n\to x$,
        \begin{equation*}
            F(x) \leq \liminf_{n\to\infty} F_n(x_n),
        \end{equation*}
    \item for every $x\in X$, there is a sequence $\{x_n\}\subset X$ with $x_n\to x$ and
        \begin{equation*}
            F(x) \geq \limsup_{n\to\infty} F_n(x_n).
        \end{equation*}
\end{enumerate}
It is a straightforward consequence of this definition that if $F_n$ $\Gamma$-converges to $F$ and $x_n$ is a minimizer of $F_n$ for every $n\in \N$, then every cluster point of the sequence $\{x_n\}$ is a minimizer to $F$. Furthermore, $\Gamma$-convergence is stable under perturbations by continuous functionals.

Here we aim to approximate optimal transport plans $\bar\pi$ of the unregularized problem for marginals $\mu$ and $\nu$ which are \emph{not} required to be in $\Llog(\Omega)$,
i.e., we allow arbitrary measures as marginals.
In this case we cannot use these marginals for the regularized problems as well, since these will admit no solutions by \cref{thm:solutionProperties}.
We therefore consider smoothed marginals $\mu_{\gamma}$ and $\nu_{\gamma}$ in $\Llog(\Omega)$ converging to $\mu$ and $\nu$, respectively, and show that the regularized problem with these marginals $\Gamma$-converges to the unregularized problem with the original marginals.
The conceptually different case of $\Gamma$-convergence for \emph{fixed, non-mollified} marginals (which then, however, need to be of finite entropy) has been treated in \cite[Thm.~2.7]{carlier2017convergence}.
Our setting with smoothed marginals allows simpler constructions in the $\limsup$-inequality since a given transport plan is merely approximated via mollification.
A further difference to \cite[Thm.~2.7]{carlier2017convergence} is that we work on a compact set $\Omega$ instead of $\mathbb{R}^n$ and need to couple the smoothing parameter to the regularization parameter to obtain $\Gamma$-convergence. 

Let $B$ be a smooth, compactly supported, nonnegative kernel with unit integral, and for $\delta>0$ and $n\in \N$ set
\begin{equation*}
    B^n_\delta(x):=\tfrac1{\delta^{n}}B(\tfrac x\delta)\,,\qquad
    G_\delta(x_1,x_2):=B^{n_1}_\delta(x_1)B^{n_2}_\delta(x_2)\,.
\end{equation*}
Since we will smooth the marginals and the transport plans by convolutions, we will need to slightly extend the domains $\Omega_{1}$ and $\Omega_{2}$ to avoid boundary effects. 
Hence, let $\tilde\Omega_1$ and $\tilde\Omega_2$ be compact supersets of $\Omega_1$ and $\Omega_2$, respectively, such that
\begin{equation*}
    \Omega_i + \supp B \subseteq \tilde\Omega_i,\quad i=1,2,
\end{equation*}
and which are large enough to contain the supports of $\mu_{\delta} := \mu\ast B_{\delta}^{n_1}$ and $\nu_{\delta}:=\nu\ast B_{\delta}^{n_2}$ for $\delta\leq 1$. (Here and in the following, we assume that the width of the convolution kernels will be small enough.)
For a function or measure $f$ on $\Omega_1$, we denote by $\tilde f$ the extension of $f$ to $\tilde\Omega_1$ by zero (and analogously for functions and measures on $\Omega_2$ and $\Omega_1\times\Omega_2$).
Let $\hat c$ be a continuous extension of $c$ onto $\tilde\Omega_1\times\tilde\Omega_2$ and set
\begin{align*}
    F_\gamma[\pi]
    &:=\int_{\tilde\Omega_1\times\tilde\Omega_2}\hat c\bd\pi+\gamma\int_{\tilde\Omega_1\times\tilde\Omega_2}\pi(\log\pi-1)\bd(x_1,x_2)\,,\\[1ex]
    E_\gamma^{\mu,\nu}[\pi]
    &:=\begin{cases}
        F_\gamma[\pi]&\text{if }\pi\in\prob(\tilde\Omega_1\times\tilde\Omega_2),\,\pushforward{(P_1)}{\pi}=\mu,\,\pushforward{(P_2)}{\pi}=\nu\,,\\
        \infty&\text{else.}
    \end{cases}
\end{align*}
Using smoothed marginals $\mu_{\delta},\nu_{\delta}$ and coupling $\gamma$ and $\delta$ in an appropriate way, we can then show $\Gamma$-convergence of $E^{\mu_{\delta},\nu_{\delta}}_{\gamma}$ to $E_0^{\mu,\nu}$ as $\gamma,\delta\to0$.
\begin{theorem}
    Let $\mu\in\prob(\Omega_{1})$, $\nu\in\prob(\Omega_{2})$, and $\gamma,\delta>0$ be such that
    \begin{equation*}
        \gamma\to0, \qquad \delta\to0, \qquad \gamma\log(\delta)\to 0,
    \end{equation*}
    which is denoted in the following by $(\gamma,\delta)\to 0$.
    Define $\mu_\delta=B_\delta^{n_1}\ast \tilde\mu$ and $\nu_\delta=B_\delta^{n_2}\ast \tilde\nu$.
    Then it holds that
    \begin{equation*}
        \gammalim_{(\gamma,\delta)\to0}E_\gamma^{\mu_\delta,\nu_\delta}=E_0^{\mu,\nu}
    \end{equation*}
    with respect to weak-$*$ convergence in $\meas(\tilde\Omega_{1}\times\tilde\Omega_{2})$.

    On the other hand, if $\gamma,\delta\to 0$ are chosen such that $\gamma\norm{\mu_\delta}_{\Phi_{\log}}\to\infty$ or $\gamma\norm{\nu_\delta}_{\Phi_{\log}}\to\infty$, then $E^{\mu_\delta,\nu_\delta}_\gamma$ does not have a finite $\Gamma$-limit. More precisely, even for a family of feasible $\pi_{\delta}$ (i.e., with marginals $\mu_{\delta}$ and $\nu_{\delta}$) it holds that
    \begin{equation*}
        \lim_{\gamma,\delta\to0} E^{\mu_\delta,\nu_\delta}_\gamma[\pi_\delta] = \infty.
    \end{equation*}
\end{theorem}
\begin{proof}
    For the first statement, we verify the two conditions in the definition of $\Gamma$-convergence.

    \emph{(i):}
    Let $\pi_\delta\weakstarto\tilde\pi$, then $\lim_{\delta\to0}F_0[\pi_\delta] = F_0[\tilde\pi]$ since $\hat c$ is continuous and bounded.
    Since $t(\log t - 1) \geq -1$, we also have that
    \begin{equation*}
        \int_{\tilde\Omega_1\times\tilde\Omega_2}\pi_\delta(\log\pi_\delta-1)\bd(x_1,x_2)\geq-|\tilde\Omega_{1}\times\tilde\Omega_{2}|
    \end{equation*}
    and thus that
    \begin{equation*}
        F_0[\tilde\pi]
        =\lim_{\delta\to0}F_0[\pi_\delta]-\lim_{\gamma\to 0}\gamma\abs{\tilde\Omega_{1}\times\tilde\Omega_{2}}
        \leq\liminf_{(\gamma,\delta)\to0}F_\gamma[\pi_\delta]\,.
    \end{equation*}
    Finally, the condition on the marginals is continuous with respect to weak-$*$ convergence of $\pi_\delta$, $\mu_\delta$, and $\nu_\delta$
    (note that $\mu_\delta,\nu_\delta\weakstarto\tilde\mu,\tilde\nu$).

    \emph{(ii):}
    It suffices to consider a recovery sequence for $\pi\in\prob(\Omega_1\times\Omega_2)$, because the marginal conditions for $\mu$ and $\nu$ can never be satisfied for $\pi\in\prob(\tilde\Omega_1\times\tilde\Omega_2)\setminus\prob(\Omega_1\times\Omega_2)$.
    If $E_0^{\mu,\nu}[\tilde\pi]=\infty$, then the $\limsup$ condition holds trivially. Let therefore $E_0^{\mu,\nu}[\tilde\pi]$ be finite.
    We set $\pi_\delta:=G_\delta\ast \tilde\pi$.
    Then $\pi_\delta\weakstarto\tilde\pi$ as well as $\pushforward{(P_1)}{\pi_\delta}=\mu_\delta$, $\pushforward{(P_2)}{\pi_\delta}=\nu_\delta$.
    Since by Young's convolution inequality $\pi_\delta \leq \norm{G_\delta}_{L^\infty}\norm{\tilde\pi}_{L^1} \leq \frac{C}{\delta^{N}}$ for some constant $C>0$ and $N:=n_1+n_2$ and we have
    \begin{equation*}
        \int_{\tilde\Omega_1\times\tilde\Omega_2} \pi_\delta\bd(x_1,x_2)= \int_{\tilde\Omega_1\times\tilde\Omega_2} \pi \ast G_{\delta}\bd(x_1,x_2)= \int_{\tilde\Omega_1\times\tilde\Omega_2} \pi\bd(x_1,x_2) \int_{\tilde\Omega_1\times\tilde\Omega_2} G_\delta\bd(x_1,x_2)= 1\,,
    \end{equation*}
    we conclude that
    \begin{equation*}
        \gamma \int_{\tilde\Omega_1\times\tilde\Omega_2}\pi_\delta(\log\pi_\delta-1)\bd(x_1,x_2)
        \leq \gamma \left(\log\Big(\frac{C}{\delta^{N}}\Big)-1\right) \int_{\tilde\Omega_1\times\tilde\Omega_2}\pi_\delta\bd(x_1,x_2)
        = -\gamma  (1+N\log \delta - \log C)\,.
    \end{equation*}
    The right-hand side vanishes for $(\gamma,\delta)\to0$ by the assumption on the (coupled) convergences of $\gamma$ and $\delta$. Hence,
    \begin{equation*}
        E_0^{\mu,\nu}[\tilde\pi]
        =\lim_{(\gamma,\delta)\to0}\left[F_0[\pi_\delta]  -\gamma  (1+N\log \delta - \log C)\right]
        \geq\limsup_{(\gamma,\delta)\to0}F_\gamma[\pi_\delta]\,.
    \end{equation*}

    For the second statement, recall from \cref{thm:projection} that
    \begin{equation*}
        \gamma\norm{\mu_\delta}_{\Phi_{\log}} \leq \max(1,\lebesgue(\Omega_2))\gamma \norm{\pi_\delta}_{\Phi_{\log}}\,,
        \qquad
        \gamma\norm{\nu_\delta}_{\Phi_{\log}} \leq \max(1,\lebesgue(\Omega_1))\gamma \norm{\pi_\delta}_{\Phi_{\log}}
    \end{equation*}
    so that $\gamma \norm{\pi_\delta}_{\Phi_{\log}}\to\infty$.    
    By \cref{thm:OrliczFunctionEstimate}, this immediately yields $\gamma\int_{\tilde\Omega_1\times\tilde\Omega_2}\pi_\delta\log^+\pi_{\delta}\bd(x_1,x_2)\to\infty$,
    which implies
    \[
        \gamma\int_{\tilde\Omega_1\times\tilde\Omega_2}\pi_\delta(\log\pi_\delta-1)\bd(x_1,x_2)
        =\gamma\int_{\tilde\Omega_1\times\tilde\Omega_2}\pi_\delta\log\pi_\delta\bd(x_1,x_2)-\gamma\to\infty
    \]
    and thus $F_\gamma[\pi_\delta]\to\infty$ so that the assertion follows.
\end{proof}

The conditions on $\gamma$ and $\delta$ are in particular satisfied for $\delta = c\gamma$ for some $c>0$.

\section{Conclusion}
\label{sec:conclusion}

In contrast to the original Kantorovich formulation of optimal transport problems, their entropic regularization is well-posed only for marginals with finite entropy. Restricting the regularized problem to such functions and applying Fenchel duality in the space $\Llog(\Omega)$ allows deriving primal-dual optimality conditions that can be interpreted pointwise almost everywhere and used to derive a continuous version of the popular Sinkhorn algorithm. For marginals that do not have finite entropy, a combined regularization and smoothing approach leads to a family of well-posed approximations that $\Gamma$-converge to the original Kantorovich formulation if the regularization and smoothing parameters are coupled in an appropriate way.

This work can be extended in several directions. For example, we have considered the usual setting where the entropic penalty is taken with respect to Lebesgue density. More general penalties have been considered in a different framework in \cite{Leonard:2008}, and other choices (such as the product measure of the marginals) are possible in the approach considered here as well and may lead to well-posedness and duality for a larger class of marginals. Naturally, a challenging but worthwhile issue would be a convergence analysis of the Sinkhorn algorithm in the considered Orlicz spaces $\Llog(\Omega)$ and $\Lexp(\Omega)$. 

\section*{Acknowledgments}

Dirk Lorenz, Hinrich Mahler, and Benedikt Wirth acknowledge support by the German Research Foundation (DFG) within the priority program ``Non-smooth and Complementarity-based Distributed Parameter Systems: Simulation and Hierarchical Optimization'' (SPP 1962) under grant numbers LO 1436/9-1 and WI 4654/1-1.
Benedikt Wirth was further supported by the Alfried Krupp Prize for Young University Teachers awarded by the Alfried Krupp von Bohlen und Halbach-Stiftung and by the DFG via Germany’s Excellence Strategy through the Cluster of Excellence ``Mathematics Münster: Dynamics – Geometry – Structure'' (EXC 2044) at the University of Münster.

The authors would also like to thank the anonymous reviewers for a number of useful comments and suggestions regarding the presentation.

\bibliographystyle{jnsao}
\bibliography{transport}

\end{document}